\documentclass[11pt]{amsart}

\usepackage[utf8]{inputenc}
\usepackage[T1]{fontenc}
\usepackage{amsmath,amssymb,amsthm,mathtools,mathrsfs}
\usepackage{enumerate,enumitem,array,booktabs}
\usepackage{graphicx,xcolor,tikz}
\usetikzlibrary{arrows.meta,calc,decorations.pathreplacing,positioning}
\usepackage[margin=1in]{geometry}
\usepackage{url}
\usepackage[colorlinks=true,linkcolor=blue,citecolor=blue,urlcolor=blue]{hyperref}

\newcommand{\Conj}{\mathrm{Conj}}
\newcommand{\Prob}{\mathbb{P}}
\newcommand{\gap}{\mathrm{gap}}
\DeclareMathOperator{\Unif}{Unif}
\newcommand{\E}{\mathbb{E}}
\newcommand{\N}{\mathbb{N}}
\newcommand{\Var}{\operatorname{Var}}

\newcommand{\PF}{\mathrm{PF}}

\newcommand{\ind}{\mathbf 1}

\newcommand{\supp}{\operatorname{supp}}

\DeclareMathOperator{\Spec}{Spec}

\DeclareMathOperator{\mix}{mix}
\DeclareMathOperator{\Fix}{Fix}

\DeclareMathOperator{\TV}{TV}

\DeclareMathOperator{\rank}{rank}
\DeclareMathOperator{\nullity}{null}

\definecolor{kcolor}{RGB}{0,107,119}
\definecolor{qcolor}{RGB}{170,55,55}
\definecolor{shiftcolor}{RGB}{176,106,0}
\definecolor{a0color}{RGB}{27,158,175}
\definecolor{a1color}{RGB}{215,83,74}
\definecolor{a2color}{RGB}{38,166,91}
\definecolor{akcolor}{RGB}{119,92,153}

\theoremstyle{plain}
\newtheorem{theorem}{Theorem}[section]
\newtheorem{lemma}[theorem]{Lemma}
\newtheorem{proposition}[theorem]{Proposition}
\newtheorem{corollary}[theorem]{Corollary}

\theoremstyle{definition}
\newtheorem{definition}[theorem]{Definition}

\newtheorem{example}[theorem]{Example}
\newtheorem{remark}[theorem]{Remark}

\title{The Dual Burnside Process}
\author{Ivan Z. Feng}
\address{Department of Mathematics, University of Southern California, Los Angeles, CA 90089-2532, USA}
\email{ifeng@usc.edu}
\urladdr{https://dornsife.usc.edu/ivan/}
\date{Received 17 November 2025; revised 10 May 2026; accepted 24 July 2026}

\begin{document}

\begin{abstract}
The Burnside process is a classical Markov chain for sampling uniformly from group orbits. We give a systematic study of the \emph{dual Burnside process}, obtained by interchanging the roles of group elements and states. This dual chain has stationary law $\pi(g)\propto |X_g|$, is reversible, and admits a matrix factorization $Q=AB$, $K=BA$ with the classical Burnside kernel $K$. As a consequence, the two chains share all nonzero eigenvalues and have mixing times that differ by at most one step. We further establish universal Doeblin floors, orbit- and conjugacy-class lumpings, exact stabilizer/fixed-set quotient pairs, and transfer principles between \(Q\) and \(K\). We analyze the explicit examples of the value-permutation model \(S_k\curvearrowright [k]^n\) and the coordinate-permutation model \(S_n\curvearrowright[k]^n\). In the value-permutation model, for fixed \(k\ge3\), the dual fixed-symbol-set quotient has \(2^k-k-1\) states, independent of \(n\), preserves the full nonzero spectrum, and has limiting nontrivial spectral radius \(1/2\). These results show that the dual chain provides both a conceptual mirror to the classical Burnside process and a genuinely useful compression mechanism for symmetry-aware Markov chain Monte Carlo.
\end{abstract}

\maketitle
\tableofcontents

\section{Introduction}\label{sec:intro}

\subsection{Sampling Up to Symmetry}

Sampling combinatorial objects with large symmetry groups is a recurring theme in probability, combinatorics, and theoretical computer science. If a finite group \(G\) acts on a finite set \(X\), write \(X/G\) for the set of \(G\)-orbits on \(X\), and set \(z:=|X/G|\). Burnside's lemma\footnote{Often attributed to Cauchy and Frobenius but popularized by Burnside; see the historical discussion in \cite{GoldbergJerrum02}.} says
\[
z=|X/G|=\frac1{|G|}\sum_{g\in G}|X_g|,
\qquad
X_g:=\{x\in X:g\cdot x=x\}.
\]
For \(x\in X\), write
\[
G_x:=\{g\in G:g\cdot x=x\}.
\]
The Burnside process, introduced by Jerrum~\cite{Jerrum93}, samples orbits
uniformly by
\[
x\longmapsto g\sim\Unif(G_x)\longmapsto y\sim\Unif(X_g),
\qquad\text{then move to }y.
\]
Its stationary law is constant on each orbit and gives mass \(1/z\) to each orbit. Thus, once the chain has mixed, projecting the sampled state to its orbit gives an approximately uniform sampler for unlabeled objects. The chain mixes rapidly in many natural settings, including value-permutation models~\cite{Paguyo22}, large sparse contingency-table models~\cite{DiaconisHowes25,DittmerThesis}, and centralizer-abelian (CA) groups\footnote{Here ``centralizer-abelian (CA) groups'' means groups in which the centralizer \(C_G(g)=\{h\in G:hg=gh\}\) is abelian for every noncentral \(g\in G\).}~\cite{Rahmani22}, but it can also be provably slow~\cite{GoldbergJerrum02}.

\subsection{The Dual Perspective}

Reversing the two-stage update yields a Markov chain on
\[
G^*:=\{g\in G:|X_g|>0\}.
\]
Specifically, from \(g\in G^*\), sample
\[
x\sim\Unif(X_g),\qquad h\sim\Unif(G_x),
\]
and move to \(h\). The transition kernel is
\[
Q(g,h)=\frac1{|X_g|}\sum_{x\in X_g\cap X_h}\frac1{|G_x|}.
\]
The matrix \(Q\) has size \(|G^*|\times |G^*|\), whereas the primal matrix \(K\) has size \(|X|\times |X|\). As an example, in \(S_k\curvearrowright[k]^n\), this means \(Q\) lives on at most \(k!\) states while \(K\) lives on \(k^n\) states; for \(k\ge2\), Section~\ref{subsec:fixed-set-and-count-quotients} sharpens this to an exact fixed-symbol-set quotient of size \(2^k-k-1\), independent of \(n\), preserving the entire nonzero spectrum. As seen explicitly in Example~\ref{ex:ternary-dual}, in the concrete case \(S_3\curvearrowright[3]^n\), the primal chain \(K\) has \(3^n\) states; by \eqref{eq:value-orbit-count}, the orbit-lumped chain has
\[
|[3]^n/S_3|=\sum_{r=0}^3S(n,r)=\frac{3^n+3}{6}
\]
states, while the dual chain \(Q\) has only 4 states and its conjugacy quotient \(\bar Q\) has 2.

With
\[
A(g,x)=\frac{\ind_{\{x\in X_g\}}}{|X_g|},
\qquad
B(x,h)=\frac{\ind_{\{h\in G_x\}}}{|G_x|},
\]
the dual kernel is \(Q=AB\), while the classical Burnside kernel is \(K=BA\).
Thus the two chains are linked by an auxiliary-variable factorization. As noted in Remark~\ref{rem:Qgg},
\[
Q(g,g)=Q(g,e)\qquad(g\in G^*),
\]
so diagonal entries are read from the \(e\)-column; together with the potentially much smaller matrix size, this structural simplification can make the dual kernel \(Q\) easier to study than the primal kernel \(K\). Although the same group-side update appears for permutation actions on words in Jerrum~\cite[Sect.~2]{Jerrum93} and Goldberg--Jerrum~\cite[Sect.~2]{GoldbergJerrum02}, to the best of our knowledge, the dual process has not previously been studied systematically for arbitrary finite group actions.

\subsection{Two Running Examples}

\begin{description}[leftmargin=2em]
\item[Value-permutation model] Here \(G=S_k\) acts on \(X=[k]^n\) by
\[
(g\cdot x)_i=g(x_i).
\]
The orbits are set partitions of \([n]\) into at most \(k\) blocks. Paguyo~\cite{Paguyo22}
proves, for \(k\ge n\),
\[
\sup_{x\in[k]^n}\|K^t(x,\cdot)-\pi_K\|_{\TV}
\le n\left(1-\frac1{2k}\right)^t,
\]
hence \(t_{\mix}(K;\varepsilon)=O(k\log(n/\varepsilon))\).

\item[Coordinate-permutation model] Here \(G=S_n\) acts on \(X=[k]^n\) by
\[
(g\cdot x)_i=x_{g^{-1}(i)}.
\]
The orbits are histograms: for \(x\in[k]^n\), the orbit is determined by the count vector
\[
(m_1(x),\dots,m_k(x)),\qquad m_a(x):=|\{i:x_i=a\}|.
\]
Equivalently, the orbits are weak \(k\)-compositions of \(n\), meaning \(k\)-tuples of nonnegative integers summing to \(n\).
Aldous~\cite{AldousFill} gives
\[
\sup_{x\in[k]^n}\|K^t(x,\cdot)-\pi_K\|_{\TV}
\le n\left(1-\frac1k\right)^t.
\]
For fixed \(k\), Diaconis~\cite{Diaconis05} gives \(n\)-independent
start-specific bounds from the all-equal start. For \(k=2\) and \(n\ge2\), Diaconis--Zhong~\cite{DiaconisZhong21} prove
\[
\frac14\left(\frac14\right)^t
\le
\|K^t(x_0,\cdot)-\pi_K\|_{\TV}
\le
4\left(\frac14\right)^t
\]
from the all-equal start \(x_0\), and Diaconis--Lin--Ram~\cite{DLR25} give the full spectral description of the unlumped binary chain.
\end{description}

\subsection{Contributions and Organization}

Section~\ref{sec:background} fixes notation and recalls the classical Burnside
kernel. The main contributions are organized around three layers.

First, Section~\ref{sec:framework} develops the action-independent theory. We
define the dual Burnside chain, prove its stationary law, establish the
factorization
\[
Q=AB,\qquad K=BA,
\]
and derive the resulting spectral correspondence, eigenvector transport, equal relaxation times, one-step total-variation transfer, \(\chi^2\)-transfer, minorization bounds, and lumping principles.

A further consequence in this section is the common-dual cover theorem: under a uniform stabilizer-preserving cover, two different actions have the same dual kernel. This gives the parking-function example: if \(\PF_n\) denotes the set of parking functions of length \(n\), viewed as a subset of \([n+1]^n\), then
Pollak's circular construction shows that the \(S_n\)-action on \(\PF_n\) has the same dual chain as the coordinate-permutation model on \([n+1]^n\).

Second, Section~\ref{sec:value} studies the value-permutation model \(S_k\curvearrowright[k]^n\). We give closed formulas for \(Q(g,h)\), construct an exact fixed-symbol-set quotient of size \(2^k-k-1\) for \(k\ge2\), and show that it preserves the full nonzero spectrum of both \(Q\) and \(K\). We also analyze the coarser fixed-point-count quotient, its support-count primal pair, and its fixed-\(k\) limiting spectrum.

Third, Section~\ref{sec:coordinate} treats the coordinate-permutation model
\(S_n\curvearrowright[k]^n\). We express \(Q(g,h)\) through the joint orbits of \(\langle g,h\rangle\), identify the dual stationary law with Ewens measure of parameter \(k\), and derive mixing and spectral consequences, including the binary spectrum, transitive-start bounds, and identity-start \(L^2\) estimates.

\section{Background and Preliminaries}\label{sec:background}

This section fixes notation for group actions, Markov kernels, total variation, minorization, and lumping and then reviews the classical Burnside kernel \(K\).

Throughout, all groups are finite, and all state spaces are finite and nonempty. In the value- and coordinate-permutation models, \(n\ge1\) unless explicitly stated otherwise. We use \((q!)^{-1}=0\) when \(q<0\).

\subsection{Group Actions}

A left action of \(G\) on \(X\) is written \(G\curvearrowright X\). For
\(x\in X\) and \(g\in G\), set
\[
O_x:=\{g\cdot x:g\in G\},\qquad
G_x:=\{g\in G:g\cdot x=x\},\qquad
X_g:=\{x\in X:g\cdot x=x\}.
\]
Orbit-stabilizer gives \(|O_x|=|G|/|G_x|\), and Burnside's lemma follows by counting fixed pairs:
\[
|X/G|=\frac1{|G|}\sum_{g\in G}|X_g|
=\frac1{|G|}\sum_{x\in X}|G_x|.
\]

\subsection{Markov-Chain Notation}

All kernels are row-stochastic. For a kernel \(P\) on a finite set \(S\) with
stationary law \(\pi\), write
\[
P_x^t:=P^t(x,\cdot),\qquad
\|\mu-\nu\|_{\TV}:=\frac12\sum_{s\in S}|\mu(s)-\nu(s)|,
\]
\[
d_P(x,t):=\|P_x^t-\pi\|_{\TV},
\qquad
d_P(t):=\max_{x\in S}d_P(x,t),
\]
and
\[
t_{\mix}(P;x,\varepsilon):=\min\{t\ge0:d_P(x,t)\le\varepsilon\},
\qquad
t_{\mix}(P;\varepsilon):=\min\{t\ge0:d_P(t)\le\varepsilon\}.
\]
Throughout the paper, unless explicitly stated otherwise, we assume
\(\varepsilon\in(0,1)\). 

A finite irreducible chain has a unique stationary distribution. If an
irreducible finite chain has a positive self-loop, then it is aperiodic.

\begin{proposition}[TV contraction; {\cite[Exercise~4.2]{LevinPeres17}}]\label{prop:TV-contraction}
For probability measures \(\mu,\nu\) on \(S\),
\[
\|\mu P-\nu P\|_{\TV}\le \|\mu-\nu\|_{\TV}.
\]
\end{proposition}

\begin{proof}
Set \(\Delta:=\mu-\nu\). Then
\[
2\|\mu P-\nu P\|_{\TV}
=
\sum_{y\in S}\left|\sum_{x\in S}\Delta(x)P(x,y)\right|
\le
\sum_{x\in S}|\Delta(x)|\sum_{y\in S}P(x,y)
=
2\|\mu-\nu\|_{\TV}. \qedhere
\]
\end{proof}

A kernel \(P\) is reversible with respect to \(\pi\) if
\[
\pi(x)P(x,y)=\pi(y)P(y,x).
\]
Equivalently, with \(\Pi:=\operatorname{diag}(\pi)\),
\begin{equation}\label{eq:rev-matrix}
\Pi P=P^\top\Pi.
\end{equation}
Equivalently,
\begin{equation}\label{eq:rev-inner}
\langle f,Pg\rangle_\pi=\langle Pf,g\rangle_\pi,
\qquad
\langle f,g\rangle_\pi:=\sum_xf(x)g(x)\pi(x).
\end{equation}
Thus \(P\) is self-adjoint on \(L^2(\pi)\). If \(\pi(x)>0\) for every \(x\in S\), then \(P\) is diagonally similar to the symmetric matrix
\[
\widetilde P:=\Pi^{1/2}P\Pi^{-1/2},
\qquad
\widetilde P^\top
=
\Pi^{-1/2}P^\top\Pi^{1/2}
=
\Pi^{-1/2}(\Pi P)\Pi^{-1/2}
=
\widetilde P.
\]
Hence \(P\) has real eigenvalues.

For probability measures \(\mu,\pi\) on \(S\), with \(\pi>0\), define
\[
\chi^2(\mu\|\pi):=\sum_{y\in S}\frac{(\mu(y)-\pi(y))^2}{\pi(y)}.
\]
For a start \(x\), write
\[
\chi_x^2(P,t)
:=
\chi^2(P^t(x,\cdot)\|\pi)
=
\sum_{y\in S}\pi(y)
\left(\frac{P^t(x,y)}{\pi(y)}-1\right)^2.
\]
By Cauchy--Schwarz,
\[
4d_P(x,t)^2\le \chi_x^2(P,t).
\]

If \(P\) is reversible, then
\begin{equation}\label{eq:chi2-return}
\chi_x^2(P,t)=\frac{P^{2t}(x,x)}{\pi(x)}-1.
\end{equation}
Indeed, reversibility of \(P^t\) gives
\[
\frac{P^t(x,y)}{\pi(y)}=\frac{P^t(y,x)}{\pi(x)},
\]
so
\[
\chi_x^2(P,t)+1
=
\sum_y\frac{P^t(x,y)^2}{\pi(y)}
=
\frac1{\pi(x)}\sum_yP^t(x,y)P^t(y,x)
=
\frac{P^{2t}(x,x)}{\pi(x)}.
\]
By \cite[Lemma~12.18]{LevinPeres17}, if \((\lambda_i,\phi_i)\) is an \(L^2(\pi)\)-orthonormal eigenbasis with \(\lambda_0=1\), \(\phi_0\equiv1\), then
\[
\chi_x^2(P,t)=\sum_{i\ge1}\lambda_i^{2t}\phi_i(x)^2.
\]
If \(P\) is reversible and irreducible, set
\[
\lambda_1(P):=\max\{\lambda:\lambda\in\Spec(P),\lambda\ne1\},
\qquad
\lambda_*(P):=\max\{|\lambda|:\lambda\in\Spec(P),\lambda\ne1\}.
\]
If there is no nontrivial eigenvalue, set
\[
\lambda_1(P)=\lambda_*(P)=0.
\]
Define
\[
\gap(P):=1-\lambda_*(P),
\qquad
t_{\mathrm{rel}}(P):=\gap(P)^{-1}.
\]

\begin{proposition}[\(L^2\)-contraction; {\cite[(12.8)]{LevinPeres17}}]
\label{prop:L2-contraction}
Assume \(P\) is reversible and irreducible with stationary law \(\pi\). For any
probability measure \(\mu\),
\[
\chi^2(\mu P^t\|\pi)\le \lambda_*(P)^{2t}\chi^2(\mu\|\pi),
\]
and hence
\[
\|\mu P^t-\pi\|_{\TV}
\le
\frac12\lambda_*(P)^t\sqrt{\chi^2(\mu\|\pi)}.
\]
In particular,
\begin{equation}\label{eq:spectral-tv}
d_P(x,t)
\le
\frac12\sqrt{\frac1{\pi(x)}-1}\,\lambda_*(P)^t.
\end{equation}
\end{proposition}

\begin{proof}
Set \(f:=\mu/\pi-1\). Then \(\pi(f)=\sum_x \pi(x)f(x)=0\), and by reversibility
\[
\frac{\mu P^t}{\pi}-1=P^tf.
\]
Applying \cite[(12.8)]{LevinPeres17},
\[
\chi^2(\mu P^t\|\pi)= \left\|\frac{\mu P^t}{\pi}-1\right\|_{L^2(\pi)}^2=\|P^tf\|_{L^2(\pi)}^2
\le \lambda_*(P)^{2t}\|f\|_{L^2(\pi)}^2
=\lambda_*(P)^{2t}\chi^2(\mu\|\pi).
\]
The TV bound follows from Cauchy--Schwarz. Taking \(\mu=\delta_x\) gives
\(\chi^2(\delta_x\|\pi)=1/\pi(x)-1\), hence \eqref{eq:spectral-tv}.
\end{proof}

\subsection{Re-indexing Principle}

\begin{lemma}[Re-indexing]\label{lemma:reindexing}
Let \(\Phi:A\to B\) be a map of finite sets. Then, for any \(F:A\to\mathbb R\),
\[
\sum_{a\in A}F(a)
=
\sum_{b\in B}\sum_{\substack{a\in A\\ \Phi(a)=b}}F(a).
\]
If \(F=G\circ\Phi\), then
\[
\sum_{a\in A}G(\Phi(a))
=
\sum_{b\in B}G(b)\,|\Phi^{-1}(b)|.
\]
\end{lemma}

\begin{proof}
Since \(A=\bigsqcup_{b\in B}\Phi^{-1}(b)\),
\[
\sum_{a\in A}F(a)
=
\sum_{b\in B}\sum_{a\in\Phi^{-1}(b)}F(a).
\]
If \(F=G\circ\Phi\), then \(F(a)=G(b)\) on \(\Phi^{-1}(b)\). Hence
\[
\sum_{a\in\Phi^{-1}(b)}F(a)=G(b)|\Phi^{-1}(b)|. \qedhere
\]
\end{proof}

\subsection{Minorization and Lumping}

For \(t_0\in\N\) with \(t_0\ge1\), \(0<\delta\le1\), and a probability measure \(\nu\), a kernel \(P\) satisfies a \((t_0,\delta,\nu)\)-minorization if
\[
P^{t_0}(x,A)\ge\delta\nu(A)
\qquad(x\in S,
A\subseteq S).
\]

\begin{proposition}[Rosenthal's bound {\cite{Rosenthal95}}]\label{prop:rosenthal}
If \(P\) satisfies a \((t_0,\delta,\nu)\)-minorization, then
\[
\|P_x^t-\pi\|_{\TV}\le(1-\delta)^{\lfloor t/t_0\rfloor}, \qquad\text{with }0^0:=1.
\]
Consequently,
\[
d_P(t)\le(1-\delta)^{\lfloor t/t_0\rfloor}.
\]
If \(0<\delta<1\), then
\[
t_{\mix}(P;\varepsilon)
\le
t_0\left\lceil\frac{\log(1/\varepsilon)}{-\log(1-\delta)}\right\rceil
\le
t_0\left\lceil\delta^{-1}\log\frac1\varepsilon\right\rceil.
\]
For \(\delta=1\), the final upper bound remains valid.
\end{proposition}

Let \(\mathcal B=\{B_1,\dots,B_m\}\) be a partition of \(S\). A kernel \(P\) is
strongly lumpable with respect to \(\mathcal B\) if
\[
\sum_{y\in B_j}P(x,y)=\sum_{y\in B_j}P(x',y)
\qquad(x,x'\in B_i,
1\le i,j\le m).
\]
Then
\begin{equation}\label{eq:lumpdef}
\bar P(B_i,B_j):=\sum_{y\in B_j}P(x,y)
\qquad(x\in B_i)
\end{equation}
is well defined. Equivalently,
\begin{equation}\label{eq:lumpave}
\bar P(B_i,B_j)=\frac1{|B_i|}\sum_{x\in B_i}\sum_{y\in B_j}P(x,y).
\end{equation}

\begin{proposition}[Basic lumping facts]\label{prop:lumping-facts}
Assume \(P\) is strongly lumpable with respect to
\(\mathcal B=\{B_1,\dots,B_m\}\), and let \(\pi\) be stationary for \(P\).
\begin{enumerate}[label=\textup{(\alph*)},itemsep=2pt]
\item The measure
\[
\bar\pi(B_i):=\sum_{x\in B_i}\pi(x)
\]
is stationary for \(\bar P\).
\item If \(P\) is reversible with respect to \(\pi\), then \(\bar P\) is reversible with respect to \(\bar\pi\). If \(P\) is irreducible, then \(\bar P\) is irreducible. If \(\bar P\) is irreducible and reversible, then
\[
\bar P(B_i,B_j)=\bar P(B_j,B_i)\quad\forall i,j
\iff
\bar\pi(B_i)\text{ is independent of }i.
\]
\item \(\Spec(\bar P)\subseteq\Spec(P)\), and every \(\bar P\)-eigenfunction lifts to a block-constant \(P\)-eigenfunction with the same eigenvalue.
\end{enumerate}
\end{proposition}

\begin{proof}
For stationarity,
\[
(\bar\pi\bar P)(B_j)
=
\sum_i\sum_{x\in B_i}\pi(x)\bar P(B_i,B_j)
=
\sum_i\sum_{x\in B_i}\pi(x)\sum_{y\in B_j}P(x,y)
=
\sum_{y\in B_j}\pi(y)
=
\bar\pi(B_j).
\]
For reversibility,
\[
\bar\pi(B_i)\bar P(B_i,B_j)
=
\sum_{x\in B_i}\sum_{y\in B_j}\pi(x)P(x,y)
=
\sum_{x\in B_i}\sum_{y\in B_j}\pi(y)P(y,x)
=
\bar\pi(B_j)\bar P(B_j,B_i).
\]
Irreducible paths project to irreducible paths. Uniform \(\bar\pi\) gives
symmetry by reversibility. Conversely, symmetry and reversibility give
\[
(\bar\pi(B_i)-\bar\pi(B_j))\bar P(B_i,B_j)=0,
\]
so irreducibility forces \(\bar\pi\) to be constant on all blocks. The spectral
claim is standard; see \cite[Lemma~12.9]{LevinPeres17}.
\end{proof}

\begin{proposition}[TV under lumping]\label{prop:TV-lumping}
Let \(\mathcal B\) be a partition of a finite set \(S\). For probability measures
\(\mu,\nu\) on \(S\), let \(\bar\mu,\bar\nu\) be their pushforwards to
\(\mathcal B\).
\begin{enumerate}[label=\textup{(\alph*)},itemsep=2pt]
\item
\[
\|\bar\mu-\bar\nu\|_{\TV}\le\|\mu-\nu\|_{\TV}.
\]
Equality holds if and only if \(\mu-\nu\) has constant sign on each block.
\item If \(\mu\) and \(\nu\) are constant on every block, then
\[
\|\mu-\nu\|_{\TV}=\|\bar\mu-\bar\nu\|_{\TV}.
\]
\end{enumerate}
\end{proposition}

\begin{proof}
Set \(\Delta:=\mu-\nu\). Then
\[
2\|\mu-\nu\|_{\TV}
=
\sum_{B\in\mathcal B}\sum_{x\in B}|\Delta(x)|,
\qquad
2\|\bar\mu-\bar\nu\|_{\TV}
=
\sum_{B\in\mathcal B}\left|\sum_{x\in B}\Delta(x)\right|.
\]
Apply the triangle inequality on each block. Equality holds exactly when
\(\Delta\) has one sign on each block. This proves \textup{(a)}. If \(\mu\) and
\(\nu\) are block constant, then \(\Delta\) is block constant and hence has one
sign on each block. Apply \textup{(a)}.
\end{proof}

\subsection{The Classical Burnside Kernel}

The classical Burnside process on \(X\) has kernel
\begin{equation}\label{eq:K-def}
K(x,y)=\frac1{|G_x|}\sum_{g\in G_x\cap G_y}\frac1{|X_g|}.
\end{equation}
Its stationary distribution is
\begin{equation}\label{eq:piK}
\pi_K(x)=\frac{|G_x|}{Z},
\qquad
Z:=\sum_{u\in X}|G_u|=|G|\,|X/G|.
\end{equation}
Equivalently, \(\pi_K(x)=1/(|X/G|\,|O_x|)\). Detailed balance is immediate:
\[
\pi_K(x)K(x,y)
=\frac1Z\sum_{g\in G_x\cap G_y}\frac1{|X_g|}
=\pi_K(y)K(y,x).
\]
Since \(e\in G_x\cap G_y\) and \(X_e=X\),
\[
K(x,y)\ge\frac1{|G_x|\,|X|}>0,
\]
so \(K\) is irreducible and aperiodic.

As mentioned, the Burnside process was introduced by Jerrum~\cite{Jerrum93} for sampling combinatorial objects modulo symmetries. Goldberg--Jerrum~\cite{GoldbergJerrum02} proved torpid-mixing examples, and Chen~\cite{Chen06} collected model-free Doeblin, lumping, and \(L^2\)-to-TV bounds. Geometric eigenvalue bounds of Diaconis--Stroock~\cite{DiaconisStroock91} and Sinclair--Jerrum~\cite{SinclairJerrum89} remain standard tools for relaxation-time estimates.

For the value-permutation model, Rahmani~\cite{RahmaniThesis} computed early formulas, and Paguyo~\cite{Paguyo22} proved rapid mixing for \(S_k\curvearrowright[k]^n\) when \(k\ge n\), together with an \(n\)-independent minorization bound in terms of \(k\). For the binary coordinate-permutation model, Diaconis--Lin--Ram~\cite{DLR25} prove
\[
\Spec(K)=\{0\}\cup
\left\{
\left(\frac{\binom{2m}{m}}{2^{2m}}\right)^2:
0\le m\le\lfloor n/2\rfloor
\right\},
\]
with multiplicity \(2^{n-1}\) at \(0\) and multiplicity \(\binom n{2m}\) at the nonzero eigenvalue indexed by \(m\).

Other variants include Rahmani's commuting chain for conjugation actions, where the primal and dual Burnside chains coincide~\cite{Rahmani22}; orbit-count estimation by Diaconis--Zhong~\cite{DiaconisZhong2025}; partition and contingency-table samplers by Diaconis--Howes~\cite{DiaconisHowes25}; lifted-inference applications~\cite{Holtzen20}; unlabeled-tree generation by
Bartholdi--Diaconis~\cite{BartholdiDiaconis24}; the flag-variety Burnside process~\cite{BurnsideFlagVariety}; Sylow double-coset cutoff~\cite{HowesSylow25}; and Dittmer's split--hyper--merge (SHM) chain for contingency tables, realized as the Burnside process for the \(H\times K\)-action
\[
(h,k)\cdot s=hsk^{-1}
\]
on \(S_n\), whose orbits are double cosets \(H\backslash S_n/K\) encoding fixed-margin contingency tables~\cite{DittmerThesis}.

\section{Main Theoretical Framework}\label{sec:framework}

This section proves the action-independent results. The mechanism is
\[
Q=AB,
\qquad
K=BA.
\]
It gives stationarity transfer, shared nonzero spectrum, eigenvector transport, equal relaxation times, one-step TV comparison, \(\chi^2\)-transfer, minorization, paired lumpings, and the common-dual cover criterion.

\subsection{Definition and Basic Properties of the Dual Chain}

\begin{definition}[Dual Burnside process]\label{def:dual}
Given a finite action \(G\curvearrowright X\), the dual Burnside process is the Markov chain on
\[
G^*:=\{g\in G:|X_g|>0\}
\]
with kernel
\begin{equation}\label{eq:Q-def}
Q(g,h)=\frac1{|X_g|}\sum_{x\in X_g\cap X_h}\frac1{|G_x|}.
\end{equation}
Equivalently, from \(g\), sample \(x\sim\Unif(X_g)\), then
\(h\sim\Unif(G_x)\), and move to \(h\).
\end{definition}

We restrict to \(G^*\) because exactly these elements can occur in the dual update. If \(X_g=\varnothing\), then \(\Unif(X_g)\) is not defined and the formula for \(Q(g,\cdot)\) gives no stochastic row. Since \(e\in G^*\), this restriction also gives irreducibility: every \(g,h\in G^*\) communicate through \(e\).

\begin{theorem}[Universal dual stationary law]\label{thm:universal-pi}
Let \(z:=|X/G|\). The probability measure
\begin{equation}\label{eq:piQ}
\pi_Q(g)=\frac{|X_g|}{|G|z},\qquad g\in G^*,
\end{equation}
is stationary and reversible for \(Q\). Moreover, \(Q\) is irreducible and
aperiodic; hence \(\pi_Q\) is unique.
\end{theorem}

\begin{proof}
Let
\[
Z:=\sum_{g\in G}|X_g|=\sum_{x\in X}|G_x|=|G|z.
\]
Then \(\sum_{g\in G^*}\pi_Q(g)=1\). For \(g,h\in G^*\),
\[
\pi_Q(g)Q(g,h)
=\frac{|X_g|}{Z}\cdot\frac1{|X_g|}
\sum_{x\in X_g\cap X_h}\frac1{|G_x|}
=\frac1Z\sum_{x\in X_g\cap X_h}\frac1{|G_x|},
\]
which is symmetric in \(g,h\). Thus detailed balance holds.

For irreducibility, if \(g,h\in G^*\), then
\[
Q(g,e)=\frac1{|X_g|}\sum_{x\in X_g}\frac1{|G_x|}>0,
\qquad
Q(e,h)=\frac1{|X|}\sum_{x\in X_h}\frac1{|G_x|}>0.
\]
Thus \(g\to e\to h\) has positive probability. Also
\[
Q(g,g)=\frac1{|X_g|}\sum_{x\in X_g}\frac1{|G_x|}>0,
\]
so every state has a self-loop.
\end{proof}

\begin{remark}\label{rem:Qgg}
The proof gives
\[
Q(g,g)=Q(g,e)>0\qquad(g\in G^*).
\]
Thus the diagonal equals the \(e\)-column. Unlike \(K\), the dual kernel may have zero entries when
\(X_g\cap X_h=\varnothing\); see Appendix~\ref{app:value-examples}.
\end{remark}

\begin{corollary}[Reversibility ratio]\label{cor:revratio}
For all \(g,h\in G^*\),
\[
|X_g|Q(g,h)=|X_h|Q(h,g).
\]
If \(Q(h,g)>0\), then
\[
\frac{Q(g,h)}{Q(h,g)}=\frac{|X_h|}{|X_g|}.
\]
In particular,
\[
\frac{Q(g,e)}{Q(e,g)}=\frac{|X|}{|X_g|}.
\]
\end{corollary}

This gives a quick way to populate the matrix: knowing one triangle, or any
single row or column, determines the corresponding reverse entries by detailed
balance.

\subsection{Primal-Dual Factorization}

Define
\begin{equation}\label{eq:legs}
A(g,x):=\frac{\ind_{\{x\in X_g\}}}{|X_g|},
\qquad
B(x,h):=\frac{\ind_{\{h\in G_x\}}}{|G_x|}.
\end{equation}
Rows of \(A\) are indexed by \(G^*\), and rows of \(B\) by \(X\). Both are row-stochastic, with rows supported on \(X_g\) and \(G_x\), respectively.

\begin{lemma}[Factorization]\label{lem:factorisations}
With \(A\) and \(B\) as in \eqref{eq:legs},
\[
Q=AB,
\qquad
K=BA.
\]
\end{lemma}

\begin{proof}
For \(g,h\in G^*\),
\[
(AB)(g,h)=\sum_{x\in X}A(g,x)B(x,h)
=\sum_{x\in X_g\cap X_h}\frac1{|X_g|\,|G_x|}=Q(g,h).
\]
For \(x,y\in X\),
\[
(BA)(x,y)=\sum_{g\in G^*}B(x,g)A(g,y)
=\sum_{g\in G_x\cap G_y}\frac1{|G_x|\,|X_g|}=K(x,y).
\]
Any contributing \(g\in G_x\cap G_y\) automatically lies in \(G^*\).
\end{proof}

Small explicit matrices illustrating \(Q=AB\), \(K=BA\), and the shared nonzero
spectrum are given in Appendices~\ref{app:value-examples} and~\ref{app:coordinate-examples}.

\begin{lemma}[$L^2$-adjointness of the legs]\label{lem:legs-adjoint}
Let
\[
\pi_K(x)=\frac{|G_x|}{Z},\qquad
\pi_Q(g)=\frac{|X_g|}{Z},\qquad
Z=|G|\,|X/G|.
\]
Then
\[
A:L^2(\pi_K)\to L^2(\pi_Q),
\qquad
B:L^2(\pi_Q)\to L^2(\pi_K)
\]
are adjoint:
\[
\langle u,Af\rangle_{\pi_Q}=\langle Bu,f\rangle_{\pi_K}
\]
for all $f:X\to\mathbb R$ and $u:G^*\to\mathbb R$. Hence
\[
K=BA=A^*A,
\qquad
Q=AB=AA^*.
\]
In particular, $K$ and $Q$ are positive semidefinite on their respective $L^2$ spaces.
\end{lemma}

\begin{proof}
Using the definitions of $A$ and $\pi_Q$,
\[
\langle u,Af\rangle_{\pi_Q}
=\sum_{g\in G^*}\frac{|X_g|}{Z}u(g)
  \sum_{x\in X}\frac{\ind_{\{x\in X_g\}}}{|X_g|}f(x)
=\frac1Z\sum_{g\in G^*}\sum_{x\in X}u(g)f(x)
\ind_{\{g\cdot x=x\}}.
\]
Using the definitions of $B$ and $\pi_K$,
\[
\langle Bu,f\rangle_{\pi_K}
=\sum_{x\in X}\frac{|G_x|}{Z}f(x)
  \sum_{g\in G^*}\frac{\ind_{\{g\in G_x\}}}{|G_x|}u(g)
=\frac1Z\sum_{x\in X}\sum_{g\in G^*}f(x)u(g)
\ind_{\{g\cdot x=x\}}.
\]
The two sums are equal. Since $K=BA$ and $Q=AB$, adjointness gives
$K=A^*A$ and $Q=AA^*$. Therefore
\[
\begin{aligned}
\langle f,Kf\rangle_{\pi_K}
&=\langle f,A^*Af\rangle_{\pi_K} \\
&=\langle Af,Af\rangle_{\pi_Q} \\
&=\|Af\|_{L^2(\pi_Q)}^2
\ge 0,
\end{aligned}
\qquad
\begin{aligned}
\langle u,Qu\rangle_{\pi_Q}
&=\langle u,AA^*u\rangle_{\pi_Q} \\
&=\langle A^*u,A^*u\rangle_{\pi_K} \\
&=\|Bu\|_{L^2(\pi_K)}^2
\ge 0.
\end{aligned}
\]
\end{proof}

\begin{theorem}[Shared nonzero spectrum]\label{thm:shared-spectrum}
\[
\Spec_{\ne0}(K)=\Spec_{\ne0}(Q),
\]
with matching algebraic multiplicities. The nonsingular parts of their Jordan
forms coincide.
\end{theorem}

\begin{proof}
This is the standard matrix fact that \(AB\) and \(BA\) have the same nonzero
eigenvalues with the same Jordan structure; see Horn--Johnson~\cite[Exercise~3.2.11]{HJ2013}.
\end{proof}

The following eigenvector transport is the eigenspace form of the \(AB/BA\)
correspondence; see also \cite[Proposition~1]{NakatsukasaLowRank}.

\begin{theorem}[Eigenvector transport]\label{prop:eig-intertwine}
Let \(\lambda\ne0\).
\begin{enumerate}[label=\textup{(\alph*)},itemsep=2pt]
\item If \(Kv=\lambda v\), then \(Av\ne0\) and \(Q(Av)=\lambda Av\).
\item If \(Qw=\lambda w\), then \(Bw\ne0\) and \(K(Bw)=\lambda Bw\).
\item On the \(\lambda\)-eigenspaces, \(A\) and \(B\) are inverse up to \(\lambda\):
\[
ABw=\lambda w,
\qquad
BAv=\lambda v.
\]
\end{enumerate}
\end{theorem}

\begin{proof}
Use \(QA=AK\) and \(KB=BQ\). If \(Av=0\), then
\(0=BAv=Kv=\lambda v\), so \(v=0\). The proof for \(B\) is identical.
\end{proof}

\begin{corollary}[Nonnegative spectrum, gaps, and relaxation times]\label{cor:gap-equal}
All eigenvalues of \(K\) and \(Q\) are real and nonnegative. Hence
\[
\lambda_*(K)=\lambda_1(K)=\lambda_1(Q)=\lambda_*(Q).
\]
Consequently,
\[
\gap(K)=\gap(Q),
\qquad
t_{\mathrm{rel}}(K)=t_{\mathrm{rel}}(Q).
\]
\end{corollary}

\begin{proof}
Both chains are reversible, hence self-adjoint, and Lemma~\ref{lem:legs-adjoint}
shows they are positive semidefinite. Thus all eigenvalues are real and
nonnegative. Theorem~\ref{thm:shared-spectrum} gives the same nonzero
nontrivial eigenvalues for \(K\) and \(Q\). Since all eigenvalues are
nonnegative, we have
\[
\lambda_*(K)=\lambda_1(K),
\qquad
\lambda_*(Q)=\lambda_1(Q).
\]
The same shared-spectrum result gives \(\lambda_1(K)=\lambda_1(Q)\). The gap
and relaxation-time identities follow from the definitions.
\end{proof}

\begin{remark}[Rank, zero eigenvalues, and singular-pair transport]
\label{rem:singular-pair-transport}
Let \(\Gamma\) be the \(G^*\times X\) matrix \( \Gamma(g,x):=\ind_{\{g\cdot x=x\}}\), and set
\[
D_G(g,g):=|X_g|,
\qquad
D_X(x,x):=|G_x|,
\qquad
N:=D_G^{-1/2}\Gamma D_X^{-1/2}.
\]
Then
\[
K\sim N^TN,
\qquad
Q\sim NN^T.
\]
Hence
\[
\rank(K)=\rank(Q)=\rank(N)=\rank(\Gamma),
\]
and
\[
\nullity(K)=|X|-\rank(N),
\qquad
\nullity(Q)=|G^*|-\rank(N).
\]

Moreover, if \(Nv=\sigma u\) and \(N^Tu=\sigma v\) with \(\sigma>0\), and
\[
\phi:=D_X^{-1/2}v,
\qquad
\psi:=D_G^{-1/2}u,
\]
then
\[
K\phi=\sigma^2\phi,
\qquad
Q\psi=\sigma^2\psi,
\qquad
A\phi=\sigma\psi,
\qquad
B\psi=\sigma\phi.
\]
Thus singular pairs of \(N\) give the nonzero eigenpairs of \(K\) and \(Q\),
together with the \(A/B\)-transport.
\end{remark}

\begin{remark}[Block-flip matrix]\label{rem:block-flip}
The stochastic matrix
\[
M=\begin{pmatrix}0&A\\ B&0\end{pmatrix}
\]
on \(G^*\sqcup X\) satisfies
\[
M^2=\begin{pmatrix}Q&0\\0&K\end{pmatrix}.
\]
Thus every nonzero eigenvalue \(\lambda\) of \(K\) or \(Q\) appears in \(M\) as
\(\pm\sqrt\lambda\), and
\[
\ker M=\ker B\oplus\ker A,
\qquad
\nullity(M)=\nullity(Q)+\nullity(K).
\]
The chain \(M\) is bipartite, hence periodic. 

One benefit of this matrix is that it packages the two kernels before multiplication: instead of forming \(AB\) and \(BA\) separately, one studies a single alternating kernel whose square contains both \(Q\) and \(K\).
\end{remark}

\subsection{Stationarity, Total Variation, and Mixing Transfer}

\begin{lemma}[Transfer of stationarity]\label{lem:stationary-transfer}
If \(\pi_K\) is stationary for \(K=BA\), then \(\pi_Q:=\pi_KB\) is stationary
for \(Q=AB\), and \(\pi_QA=\pi_K\). Conversely, if \(\pi_Q\) is stationary for
\(Q\), then \(\pi_K:=\pi_QA\) is stationary for \(K\), and \(\pi_KB=\pi_Q\).
\end{lemma}

\begin{proof}
If \(\pi_KK=\pi_K\), then
\[
\pi_QQ=\pi_KBAB=\pi_KKB=\pi_KB=\pi_Q,
\qquad
\pi_QA=\pi_KBA=\pi_KK=\pi_K.
\]
The converse is identical.
\end{proof}

\begin{remark}[Recovering the explicit stationary laws]\label{rem:recover-stationary-laws}
For the Burnside legs, Lemma~\ref{lem:stationary-transfer} gives the explicit
laws directly. With \(Z=|G|\,|X/G|\),
\[
(\pi_KB)(g)
=
\sum_{x\in X_g}\frac{|G_x|}{Z}\frac1{|G_x|}
=
\frac{|X_g|}{Z}
=
\pi_Q(g),
\]
and
\[
(\pi_QA)(x)
=
\sum_{g\in G_x}\frac{|X_g|}{Z}\frac1{|X_g|}
=
\frac{|G_x|}{Z}
=
\pi_K(x).
\]
\end{remark}

For \(t\ge1\),
\begin{equation}\label{eq:Q-AK}
Q^t=AK^{t-1}B,
\qquad
K^t=BQ^{t-1}A.
\end{equation}

\begin{theorem}[Pointwise transfer inequality]\label{lem:ptwise-transfer}
For \(g\in G^*\), \(x\in X\), and \(t\ge1\),
\[
d_Q(g,t)\le \max_{y\in X_g}d_K(y,t-1),
\qquad
d_K(x,t)\le \max_{h\in G_x}d_Q(h,t-1).
\]
\end{theorem}

\begin{proof}
Let \(A_g:=A(g,\cdot)\) denote the \(g\)-th row of \(A\).
By \eqref{eq:Q-AK} and Lemma~\ref{lem:stationary-transfer},
\[
Q^t(g,\cdot)-\pi_Q=(A_gK^{t-1}-\pi_K)B.
\]
By Proposition~\ref{prop:TV-contraction},
\[
d_Q(g,t)\le\|A_gK^{t-1}-\pi_K\|_{\TV}.
\]
Since \(A_g\) is uniform on \(X_g\),
\[
A_gK^{t-1}-\pi_K
=\sum_{y\in X_g}A(g,y)(K^{t-1}(y,\cdot)-\pi_K).
\]
The triangle inequality gives
\[
d_Q(g,t)\le\max_{y\in X_g}d_K(y,t-1).
\]
The second bound follows identically from \(K^t=BQ^{t-1}A\) and
\(\pi_K=\pi_QA\).
\end{proof}

\begin{corollary}[One-step comparison]\label{lem:TV-step}
For \(t\ge1\),
\[
d_Q(t)\le d_K(t-1),
\qquad
d_K(t)\le d_Q(t-1).
\]
\end{corollary}

\begin{proof}
Take the maximum over \(g\in G^*\) and \(x\in X\) in
Theorem~\ref{lem:ptwise-transfer}.
\end{proof}

The alternating-walk interpretation of this one-step comparison is illustrated
in Fig.~\ref{fig:primal-dual}.

\begin{theorem}[Mixing time equivalence]\label{thm:mixt}
We have
\[
|t_{\mix}(Q;\varepsilon)-t_{\mix}(K;\varepsilon)|\le1.
\]
\end{theorem}

\begin{proof}
If \(d_K(s)\le\varepsilon\), then \(d_Q(s+1)\le\varepsilon\). Thus \(t_{\mix}(Q;\varepsilon)\le t_{\mix}(K;\varepsilon)+1\). The other inequality
is symmetric.
\end{proof}

The one-step comparison transfers cutoff statements immediately; for example,
Howes~\cite{HowesSylow25} proves cutoff for a Burnside process on Sylow double
cosets, hence the corresponding dual chain has the same cutoff location.

\begin{corollary}[Cutoff equivalence]\label{cor:cutoff-equivalence}
For any sequence of finite actions \(G_n\curvearrowright X_n\), assume \(t_{\mix}(K_n;\varepsilon_0)\to\infty\) for some \(\varepsilon_0\in(0,1)\).
Then \((K_n)\) has cutoff at location \(a_n\), meaning
\[
t_{\mix}(K_n;\varepsilon)=a_n(1+o(1))
\]
iff \((Q_n)\) has cutoff at the same location \(a_n\). If
\[
t_{\mix}(K_n;\varepsilon)=a_n+O(w_n),
\]
then the same window \(w_n\) works for \((Q_n)\) when \(w_n\to\infty\); bounded windows are shared up to an \(O(1)\) enlargement.
\end{corollary}

\begin{proof}
By Theorem~\ref{thm:mixt}, the two mixing times differ by at most one for every
fixed \(\varepsilon\in(0,1)\). Since \(a_n\to\infty\), this preserves cutoff location. The window statement follows by adding an \(O(1)\) error.
\end{proof}

\begin{corollary}[Start-specific transfer]\label{cor:short-transfer}
Fix \(x_0\in X\), \(0<\alpha<1\), \(0<C_1\le C_2<\infty\), and assume
\(C_2\ge\varepsilon\) and
\[
C_1\alpha^t\le d_K(x_0,t)\le C_2\alpha^t\qquad(t\ge0).
\]
Then
\[
\left\lceil\frac{\log(C_1/\varepsilon)}{-\log\alpha}\right\rceil
\le t_{\mix}(K;x_0,\varepsilon)\le
\left\lceil\frac{\log(C_2/\varepsilon)}{-\log\alpha}\right\rceil,
\qquad
d_Q(t)\ge C_1\alpha^{t+1}\quad(t\ge0),
\]
and hence
\[
t_{\mix}(Q;\varepsilon)
\ge
\left\lceil\frac{\log(C_1/\varepsilon)}{-\log\alpha}\right\rceil-1.
\]
If \(h_0\in G_{x_0}\) and \(d_K(x,u)\le C_2\alpha^u\) for all \(u\ge0\) and \(x\in X_{h_0}\), then
\[
d_Q(h_0,t)\le C_2\alpha^{t-1}\quad(t\ge1),
\qquad
t_{\mix}(Q;h_0,\varepsilon)
\le
1+\left\lceil\frac{\log(C_2/\varepsilon)}{-\log\alpha}\right\rceil .
\]
\end{corollary}

\begin{proof}
The lower bound for \(t_{\mix}(K;x_0,\varepsilon)\) follows by solving
\(C_1\alpha^t\le\varepsilon\), while the upper bound follows by solving
\(C_2\alpha^t\le\varepsilon\); the assumption \(C_2\ge\varepsilon\) ensures
the upper-bound time is nonnegative. Since \(d_K(t+1)\le d_Q(t)\),
\[
d_Q(t)\ge d_K(x_0,t+1)\ge C_1\alpha^{t+1},
\]
and solving \(C_1\alpha^{t+1}\le\varepsilon\) gives the stated lower bound for \(t_{\mix}(Q;\varepsilon)\). Finally, Theorem~\ref{lem:ptwise-transfer} gives
\[
d_Q(h_0,t)\le \max_{x\in X_{h_0}}d_K(x,t-1)\le C_2\alpha^{t-1},
\]
and solving \(C_2\alpha^{t-1}\le\varepsilon\) gives the stated upper bound.
\end{proof}

\begin{figure}[htbp]
    \centering
    \resizebox{.9\linewidth}{!}{%
    \begin{tikzpicture}[
      >=Latex,
      every node/.style={font=\small},
      leg/.style={->,line width=0.8pt},
      twostep/.style={->,line width=1.0pt},
      comparison/.style={<->,line width=1.0pt}
    ]
      \node (x1) at (0,0) {$x_1$};
      \node (g1) at (1.35,0) {$g_1$};
      \node (x2) at (2.70,0) {$x_2$};
      \node (g2) at (4.05,0) {$g_2$};
      \node (x3) at (5.40,0) {$x_3$};
      \node (g3) at (6.75,0) {$g_3$};
      \node (dots) at (7.90,0) {$\cdots$};
      \node (xinf) at (9.15,0) {$x_\infty$};
      \node (ginf) at (10.55,0) {$g_\infty$};

      \draw[leg] (x1) -- node[above] {$B$} (g1);
      \draw[leg] (g1) -- node[above] {$A$} (x2);
      \draw[leg] (x2) -- node[above] {$B$} (g2);
      \draw[leg] (g2) -- node[above] {$A$} (x3);
      \draw[leg] (x3) -- node[above] {$B$} (g3);
      \draw[leg] (g3) -- (dots);
      \draw[leg] (dots) -- node[above] {$A$} (xinf);
      \draw[leg] ($(xinf.east)+(0,0.12)$) -- node[above] {$B$}
        ($(ginf.west)+(0,0.12)$);
      \draw[leg] ($(ginf.west)+(0,-0.12)$) -- node[below] {$A$}
        ($(xinf.east)+(0,-0.12)$);

      \node[below=2pt of xinf] {$\pi_K$};
      \node[below=2pt of ginf] {$\pi_Q$};

      \draw[twostep,kcolor] (x1.south) to[bend right=32]
        node[below] {$K=BA$} (x2.south);
      \draw[twostep,kcolor] (x2.south) to[bend right=32]
        node[below] {$K$} (x3.south);
      \draw[twostep,qcolor,dashed] (g1.south) to[bend right=43]
        node[below=2pt] {$Q=AB$} (g2.south);
      \draw[twostep,qcolor,dashed] (g2.south) to[bend right=43]
        node[below=2pt] {$Q$} (g3.south);

      \coordinate (kxleft) at ($(x1)+(0,1.05)$);
      \coordinate (kxright) at ($(xinf)+(0,1.05)$);
      \draw[comparison,kcolor] (kxleft) -- node[fill=white,inner sep=2pt]
        {$d_K(t)$} (kxright);
      \draw[kcolor,line width=0.8pt] (x1.north) -- (kxleft);
      \draw[kcolor,line width=0.8pt] (xinf.north) -- (kxright);

      \coordinate (qleft) at ($(g1)+(0,1.75)$);
      \coordinate (qright) at ($(ginf)+(0,1.75)$);
      \draw[comparison,qcolor,dashed] (qleft) -- node[fill=white,inner sep=2pt]
        {$d_Q(t)$} (qright);
      \draw[qcolor,dashed,line width=0.8pt] (g1.north) -- (qleft);
      \draw[qcolor,dashed,line width=0.8pt] (ginf.north) -- (qright);

      \coordinate (shiftleft) at ($(g2)+(0,2.45)$);
      \coordinate (shiftright) at ($(ginf)+(0,2.45)$);
      \draw[comparison,shiftcolor,dashed] (shiftleft) --
        node[fill=white,inner sep=2pt] {$d_Q(t-1)$} (shiftright);
      \draw[shiftcolor,dashed,line width=0.8pt] (g2.north) -- (shiftleft);
      \draw[shiftcolor,dashed,line width=0.8pt] (ginf.north) -- (shiftright);
    \end{tikzpicture}%
    }
    \caption{Alternating walk on \(G^*\sqcup X\). Two consecutive legs on
    \(G^*\) give \(Q=AB\), and two consecutive legs on \(X\) give \(K=BA\).}
    \label{fig:primal-dual}
\end{figure}

\begin{theorem}[Exact \(\chi^2\) transfer]\label{thm:exact-chi2-transfer}
Let \(\lambda_1\) be the largest nontrivial eigenvalue of \(K\), equivalently of \(Q\). For \(g\in G^*\) and \(t\ge1\),
\[
\chi_g^2(Q,t)
=
\sum_{x\in X_g}\sum_{y\in X_g}
\frac1{|X_g|^2}
\left(\frac{K^{2t-1}(x,y)}{\pi_K(y)}-1\right).
\]
Consequently,
\[
\chi_g^2(Q,t)
\le
\lambda_1\frac1{|X_g|}\sum_{x\in X_g}\chi_x^2(K,t-1)
\le
\lambda_1\max_{x\in X_g}\chi_x^2(K,t-1).
\]
\end{theorem}

\begin{proof}
By \eqref{eq:chi2-return},
\[
\chi_g^2(Q,t)=\frac{Q^{2t}(g,g)}{\pi_Q(g)}-1.
\]
Recall the normalizing constant \(Z=|G|\,|X/G|\). Using \(Q^{2t}=AK^{2t-1}B\) and
\[
\pi_Q(g)A(g,y)=\pi_K(y)B(y,g)=\frac1Z\ind_{\{y\in X_g\}},
\]
we get
\[
\frac{Q^{2t}(g,g)}{\pi_Q(g)}
=
\sum_{x,y}A(g,x)A(g,y)\frac{K^{2t-1}(x,y)}{\pi_K(y)}.
\]
Subtracting \(1=\sum_{x,y}A(g,x)A(g,y)\), and using that \(A(g,\cdot)\) is
uniform on \(X_g\), gives the identity.

Let \(\{\phi_i\}\) be an \(L^2(\pi_K)\)-orthonormal eigenbasis for \(K\), with
\(\phi_0\equiv1\). Then
\[
\frac{K^s(x,y)}{\pi_K(y)}-1
=
\sum_{i\ge1}\lambda_i^s\phi_i(x)\phi_i(y).
\]
With \(s=2t-1\),
\[
\chi_g^2(Q,t)
=
\sum_{i\ge1}\lambda_i^{2t-1}
\left(\frac1{|X_g|}\sum_{x\in X_g}\phi_i(x)\right)^2.
\]
For \(t=1\), \(\lambda_i^{2t-1}\le \lambda_1\) is just
\(\lambda_i\le\lambda_1\). For \(t\ge2\),
\[
\lambda_i^{2t-1}\le \lambda_1\lambda_i^{2t-2},
\]
since \(0\le\lambda_i\le\lambda_1\). Jensen's inequality gives
\[
\left(\frac1{|X_g|}\sum_{x\in X_g}\phi_i(x)\right)^2
\le
\frac1{|X_g|}\sum_{x\in X_g}\phi_i(x)^2.
\]
Substitute and use the spectral expansion of \(\chi_x^2(K,t-1)\).
\end{proof}

\begin{corollary}[Pointwise \(\chi^2\) transfer]\label{cor:chi2-transfer}
For \(g\in G^*\) and \(t\ge1\),
\[
\chi_g^2(Q,t)\le\max_{x\in X_g}\chi_x^2(K,t-1).
\]
\end{corollary}

\begin{proof}
Use \(\lambda_1\le1\) in Theorem~\ref{thm:exact-chi2-transfer}.
\end{proof}

\begin{remark}[Other divergences]
The same proof as Theorem~\ref{lem:ptwise-transfer} works for any finite
\(f\)-divergence
\[
D_\Phi(\mu\|\nu)
=
\sum_s\nu(s)\Phi\!\left(\frac{\mu(s)}{\nu(s)}\right),
\]
where \(\Phi\) is convex and \(\Phi(1)=0\). Thus, for \(t\ge1\),
\[
D_\Phi(Q^t(g,\cdot)\|\pi_Q)
\le
\max_{x\in X_g}D_\Phi(K^{t-1}(x,\cdot)\|\pi_K),
\]
and similarly
\[
D_\Phi(K^t(x,\cdot)\|\pi_K)
\le
\max_{h\in G_x}D_\Phi(Q^{t-1}(h,\cdot)\|\pi_Q).
\]
This includes relative entropy, total variation, and \(\chi^2\)-distance. The stronger return-probability identity in Theorem~\ref{thm:exact-chi2-transfer} is special to \(\chi^2\) and reversibility.
\end{remark}

\subsection{Universal Doeblin Floors}\label{subsec:universal-minorization}

Let
\[
M:=\max_{x\in X}|G_x|.
\]
Chen~\cite{Chen06} proves the model-free floor \(K(x,\cdot)\ge |G|^{-1}\pi_K(\cdot)\). The next
floor improves it when \(M<|G|\), and gives the dual analogue.

\begin{lemma}[Uniform floors]\label{lem:universal-floors}
For every finite action \(G\curvearrowright X\):
\begin{enumerate}[label=\textup{(\alph*)},itemsep=2pt]
\item \(Q(g,\cdot)\ge M^{-1}\delta_e(\cdot)\) for \(g\in G^*\).
\item \(K(x,\cdot)\ge M^{-1}\Unif(X)(\cdot)\) for \(x\in X\).
\end{enumerate}
\end{lemma}

\begin{proof}
Since \(e\in G_x\) for all \(x\),
\[
Q(g,e)=\frac1{|X_g|}\sum_{x\in X_g}\frac1{|G_x|}\ge\frac1M.
\]
For \(K\), the identity contributes to every transition:
\[
K(x,y)\ge\frac1{|G_x|\,|X_e|}=\frac1{|G_x|\,|X|}\ge\frac1{M|X|}.
\]
For any \(C\subseteq X\), summing over \(y\in C\) gives the stated minorization.
\end{proof}

\begin{theorem}[Universal mixing bound]\label{cor:universal-mix}
For all \(t\ge1\),
\[
d_Q(t)\le(1-M^{-1})^t,
\qquad
d_K(t)\le(1-M^{-1})^t.
\]
Thus
\[
t_{\mix}(Q;\varepsilon),\ t_{\mix}(K;\varepsilon)
\le
\left\lceil M\log\frac1\varepsilon\right\rceil.
\]
\end{theorem}

\begin{proof}
Apply Proposition~\ref{prop:rosenthal} with \(t_0=1\) and \(\delta=M^{-1}\).
\end{proof}

Although the one-step TV comparison gives sharper TV bounds, minorization itself
also transfers through the factorization at the cost of two steps.

\begin{theorem}[Two-step minorization transfer]\label{prop:two-step-transfer}
If \(K\ge\delta\nu\) row-wise on \(X\), then
\[
Q^2\ge\delta(\nu B)
\]
row-wise on \(G^*\). Hence
\[
d_Q(t)\le(1-\delta)^{\lfloor t/2\rfloor}.
\]
\end{theorem}

\begin{proof}
Since \(K\ge\delta\nu\) and \(B\) is stochastic, \(KB\ge\delta(\nu B)\). Then
\[
Q^2(g,\cdot)=(AKB)(g,\cdot)=\sum_xA(g,x)(KB)(x,\cdot)\ge\delta(\nu B).
\]
Apply Proposition~\ref{prop:rosenthal} with \(t_0=2\).
\end{proof}

\subsection{Lumping Principle}\label{sec:lumpingp}

\subsubsection{General lumping and TV comparison}
\label{subsubsec:general-lumping-TV}

Recall \(Q=AB\) and \(K=BA\).

\begin{lemma}[Orbit-invariance propagates]\label{lem:row-orbit-invariance}
Let \(G\) act on \(S\), and let \(P\) be \(G\)-equivariant:
\[
P(gs,gt)=P(s,t).
\]
If \(P(s,\cdot)\) is orbit-constant, i.e.,
\[
P(s,gz)=P(s,z)\qquad(g\in G,\ z\in S),
\]
then \(P^t(s,\cdot)\) is orbit-constant for every \(t\ge1\).
\end{lemma}

\begin{proof}
The case \(t=1\) is the hypothesis. Assume \(P^t(s,\cdot)\) is orbit-constant.
Then
\[
\begin{aligned}
P^{t+1}(s,gz)
&=\sum_{u\in S}P^t(s,u)P(u,gz)\\
&=\sum_{u\in S}P^t(s,u)P(g^{-1}u,z)
\qquad\text{by \(G\)-equivariance}\\
&=\sum_{v\in S}P^t(s,gv)P(v,z)
\qquad (v=g^{-1}u)\\
&=\sum_{v\in S}P^t(s,v)P(v,z)
=P^{t+1}(s,z).
\end{aligned} 
\]
\end{proof}

\begin{theorem}[Equivariant lumping and TV comparison]\label{thm:eq-lumping-TV}
Let \(G\) act on \(S\), and let \(P\) be \(G\)-equivariant (\(P(au,av)=P(u,v)\) for all \(a\in G\)). Then \(P\) is
strongly lumpable by \(G\)-orbits. Writing \([s]\) for the orbit of \(s\),
\[
\bar P([s],[t])=\sum_{u\in[t]}P(s,u)
\]
is well defined. If \(\pi\) is stationary for \(P\), then the pushforward
\(\bar\pi([s])=\sum_{u\in[s]}\pi(u)\) is stationary for \(\bar P\). Moreover,
\[
\|P^t(s,\cdot)-\pi\|_{\TV}
\ge
\|\bar P^t([s],\cdot)-\bar\pi\|_{\TV}.
\]
If \(\pi\) is orbit-constant (\(\pi(au)=\pi(u)\) for all \(a\in G\)) and
\(P(s,\cdot)\) is orbit-constant (\(P(s,au)=P(s,u)\) for all \(a\in G\)), then equality holds for all \(t\ge1\).
\end{theorem}

\begin{proof}
If \(s'=gs\), then
\[
\sum_{u\in[t]}P(s',u)
=\sum_{u\in[t]}P(gs,u)
=\sum_{u\in[t]}P(s,g^{-1}u)
=\sum_{v\in[t]}P(s,v).
\]
Thus the lumping is well defined. Stationarity is Proposition~\ref{prop:lumping-facts}\textup{(a)}.
The TV inequality is Proposition~\ref{prop:TV-lumping}\textup{(a)}. If
\(P(s,\cdot)\) is orbit-constant, Lemma~\ref{lem:row-orbit-invariance} gives
orbit-constancy of \(P^t(s,\cdot)\) for \(t\ge1\); then
Proposition~\ref{prop:TV-lumping}\textup{(b)} gives equality.
\end{proof}

\begin{lemma}[Subset aggregation]\label{lem:subset-agg}
For \(\Omega\subseteq X\) and \(S\subseteq G^*\),
\[
(A\mathbf1_\Omega)(g)=\frac{|X_g\cap\Omega|}{|X_g|},
\qquad
(B\mathbf1_S)(x)=\frac{|G_x\cap S|}{|G_x|}.
\]
Hence
\[
\sum_{y\in\Omega}K(x,y)
=
\frac1{|G_x|}\sum_{h\in G_x}\frac{|X_h\cap\Omega|}{|X_h|},
\]
and
\[
\sum_{h\in S}Q(g,h)
=
\frac1{|X_g|}\sum_{x\in X_g}\frac{|G_x\cap S|}{|G_x|}.
\]
\end{lemma}

\begin{proof}
The first two identities are the definitions of \(A\) and \(B\); the last two are
\(K=BA\) and \(Q=AB\).
\end{proof}

\begin{lemma}[Conjugacy and orbit invariance]\label{lem:inv-block}
For \(a,g\in G\) and \(x\in X\),
\[
G_{ax}=aG_xa^{-1},
\qquad
X_{aga^{-1}}=aX_g.
\]
Consequently, \(|G_{ax}\cap C|=|G_x\cap C|\) for every conjugacy class \(C\),
and \(|X_{aga^{-1}}\cap O|=|X_g\cap O|\) for every orbit \(O\subseteq X\).
\end{lemma}

\begin{proof}
\[
t\in G_{ax}
\iff
 ta x=ax
\iff
a^{-1}ta\in G_x.
\]
Thus \(G_{ax}=aG_xa^{-1}\). Similarly,
\[
y\in X_{aga^{-1}}
\iff
a^{-1}y\in X_g,
\]
so \(X_{aga^{-1}}=aX_g\). The final claims follow because conjugation preserves
conjugacy classes and the \(G\)-action preserves orbits.
\end{proof}

\begin{corollary}[Orbit lumping for \(K\)]\label{cor:lump-K-TV}
The Burnside kernel \(K\) is strongly lumpable by \(G\)-orbits in \(X\). Its
lumped kernel is
\begin{equation}\label{eq:Kbarorb}
\bar K([x],[y])
=
\frac1{|G_x|}\sum_{h\in G_x}\frac{|X_h\cap[y]|}{|X_h|}.
\end{equation}
The lumped stationary law is uniform:
\[
\bar\pi_K([x])=\frac1{|X/G|}.
\]
Moreover, \(\bar K\) is reversible, irreducible, and symmetric. For all \(t\),
\[
d_K(x,t)\ge d_{\bar K}([x],t).
\]
If \(K(x,\cdot)\) is orbit-constant, then equality holds for all \(t\ge1\).
If \(G\) is abelian, equality holds for every start and every \(t\ge1\).
\end{corollary}

\begin{proof}
Lemma~\ref{lem:inv-block} gives \(K(ax,ay)=K(x,y)\). Then we apply
Theorem~\ref{thm:eq-lumping-TV}. Formula \eqref{eq:Kbarorb} is
Lemma~\ref{lem:subset-agg} with \(\Omega=[y]\). Since
\(\pi_K(x)=1/(|X/G|\,|[x]|)\), the pushforward is uniform. By
Proposition~\ref{prop:lumping-facts}, reversibility and irreducibility transfer;
uniform stationary law gives symmetry. If \(G\) is abelian, then \(G_{az}=aG_za^{-1}=G_z\),
hence \(K(x,az)=K(x,z)\).
\end{proof}

Orbit lumping for \(K\) is standard in Burnside-process analyses; see, for
example, Paguyo~\cite{Paguyo22} for the value-permutation model, and
Diaconis~\cite{Diaconis05} and Diaconis--Zhong~\cite{DiaconisZhong21} for
coordinate-permutation models.

\begin{example}[Starts where orbit lumping preserves TV]\label{ex:orbit-TV-equality}
\leavevmode
\begin{enumerate}[label=\textup{(\alph*)},itemsep=2pt]
\item In the value-permutation model \(S_k\curvearrowright[k]^n\), if \(x\)
uses all \(k\) symbols, then \(G_x=\{e\}\), so
\[
K(x,\cdot)=\Unif([k]^n).
\]
Hence \(K(x,\cdot)\) is orbit-constant, and Corollary~\ref{cor:lump-K-TV}
gives
\[
d_K(x,t)=d_{\bar K}([x],t)\qquad(t\ge1).
\]

\item In the coordinate-permutation model \(S_n\curvearrowright[k]^n\), if
\(x=a^n\) is constant, then \(G_x=S_n\), and
\[
K(x,z)=\frac1{n!}\sum_{g\in G_z}\frac1{|X_g|}.
\]
Thus \(K(x,z)\) depends only on the orbit of \(z\), so
\[
d_K(x,t)=d_{\bar K}([x],t)\qquad(t\ge1).
\]
\end{enumerate}
\end{example}

\begin{proposition}[Chen's orbit-chain bound]\label{prop:chen-orbit-only}
For the orbit-lumped Burnside chain, and \(t\ge1\),
\[
d_{\bar K}(t)\le\left(1-\frac1{|X|}\right)^t,
\qquad
t_{\mix}(\bar K;\varepsilon)
\le
\left\lceil |X|\log\frac1\varepsilon\right\rceil .
\]
Whenever TV is preserved by orbit lumping from \(x\), the same bound holds for
\(K\) from \(x\) for \(t\ge1\).
\end{proposition}

\begin{proof}
This is Chen's coupling bound~\cite[Proposition~11, Section~5]{Chen06}, plus
Corollary~\ref{cor:lump-K-TV}.
\end{proof}

\begin{corollary}[Conjugacy lumping for \(Q\)]\label{cor:lump-Q-TV}
The dual kernel \(Q\) is strongly lumpable by conjugacy classes in \(G^*\). Its
lumped kernel is
\begin{equation}\label{eq:Qbarconj}
\bar Q([g],[h])
=
\frac1{|X_g|}\sum_{x\in X_g}\frac{|G_x\cap[h]|}{|G_x|}.
\end{equation}
The lumped stationary law is
\[
\bar\pi_Q([g])=\frac{|[g]|\,|X_g|}{|G|\,|X/G|}.
\]
Moreover, \(\bar Q\) is reversible and irreducible. For all \(t\),
\[
d_Q(g,t)\ge d_{\bar Q}([g],t).
\]
If \(Q(g,\cdot)\) is class-constant, then equality holds for all \(t\ge1\). This
holds in particular when \(g\in Z(G)\).
\end{corollary}

\begin{proof}
Lemma~\ref{lem:inv-block} gives \(Q(aga^{-1},aha^{-1})=Q(g,h)\). Then we apply Theorem~\ref{thm:eq-lumping-TV}. Formula \eqref{eq:Qbarconj} is
Lemma~\ref{lem:subset-agg} with \(S=[h]\). The stationary formula follows from \(\pi_Q(g)=|X_g|/(|G|\,|X/G|)\). If \(g\in Z(G)\), then \(aX_g=X_g\), and the
change of variables \(x=ay\) gives \(Q(g,aha^{-1})=Q(g,h)\).
\end{proof}

\begin{example}[Starts where conjugacy lumping preserves TV]
\label{ex:conj-TV-equality}
\leavevmode
\begin{enumerate}[label=\textup{(\alph*)},itemsep=2pt]
\item For every action,
\[
Q(e,h)=\frac1{|X|}\sum_{x\in X_h}\frac1{|G_x|}
\]
is constant on conjugacy classes of \(h\). Hence, for the conjugacy quotient \(\bar Q\),
\[
d_Q(e,t)=d_{\bar Q}([e],t)\qquad(t\ge1).
\]

\item In the coordinate-permutation model \(S_n\curvearrowright[k]^n\), if
\(g\) is an \(n\)-cycle, then \(Q(g,\cdot)\equiv 1/n!\). Hence
\[
d_Q(g,t)=d_{\bar Q}([g],t)\qquad(t\ge1).
\]
\end{enumerate}
\end{example}

\subsubsection{Auxiliary-variable scheme for lumped kernels}
\label{subsec:paired-lumpings}

This part packages the quotient structures that come from the legs. It is an auxiliary-variable statement; Gibbs sampling is a classical example \cite{DKSC08}. The twisted Burnside process of Diaconis--Zhong~\cite{DiaconisZhong21} also has the same form, with legs
\[
A_v(g,x)=\frac{v(x)\ind_{\{x\in X_g\}}}{\sum_{u\in X_g}v(u)},
\qquad
B_w(x,h)=\frac{w(h)\ind_{\{h\in G_x\}}}{\sum_{u\in G_x}w(u)}.
\]
Here \(v\colon X\to(0,\infty)\) and \(w\colon G\to(0,\infty)\) are
strictly positive weight functions.

All algebraic consequences of
\[
Q=AB,\qquad K=BA
\]
extend to any finite auxiliary-variable pair \(A:\mathcal Y\to X\), \(B:X\to\mathcal Y\) of compatible sizes; in particular, Theorems~\ref{thm:shared-spectrum} and~\ref{prop:eig-intertwine} remain valid.
If \(A\) and \(B\) are row-stochastic, then \(Q\) and \(K\) are Markov kernels, Lemma~\ref{lem:stationary-transfer} applies, and Theorem~\ref{prop:two-step-transfer} remains valid. For statements involving \(d_P\) or \(t_{\mix}\), fix paired stationary laws \(\pi_{\mathcal Y},\pi_X\) satisfying
\[
\pi_{\mathcal Y}A=\pi_X,\qquad \pi_XB=\pi_{\mathcal Y},
\]
as in Lemma~\ref{lem:stationary-transfer}. With respect to these laws, the pointwise transfer theorem becomes
\[
d_Q(y,t)\le \max_{x:A(y,x)>0}d_K(x,t-1),
\qquad
d_K(x,t)\le \max_{y:B(x,y)>0}d_Q(y,t-1)
\]
for \(t\ge1\). Consequently Corollary~\ref{lem:TV-step} holds. If one of \(t_{\mix}(Q;\varepsilon)\) and \(t_{\mix}(K;\varepsilon)\) is finite, then so is the other, and Theorem~\ref{thm:mixt} gives the one-step comparison. Results using adjointness, reversibility, positive semidefiniteness, block-flip nullity, or \(\chi^2\)-return identities require the corresponding extra structure.

\begin{definition}[Compatible paired partitions]
\label{def:compatible-paired-partitions}
Let \(A:Y\to X\) and \(B:X\to Y\) be row-stochastic, and set \(Q:=AB\), \(K:=BA\). Let \(\mathcal C=\{C_\alpha\}\) partition \(Y\), and \(\mathcal D=\{D_i\}\) partition \(X\). The pair \((\mathcal C,\mathcal D)\) is compatible with \(A,B\) if, for all \(\alpha,i\),
\[
\sum_{x\in D_i}A(y,x)\text{ is independent of }y\in C_\alpha,
\qquad
\sum_{y\in C_\alpha}B(x,y)\text{ is independent of }x\in D_i.
\]
\end{definition}

\begin{theorem}[Paired quotient theorem]
\label{thm:paired-lumpings-general}
Assume \((\mathcal C,\mathcal D)\) is compatible. Define
\[
\bar A(C_\alpha,D_i):=\sum_{x\in D_i}A(y,x)\quad(y\in C_\alpha),
\qquad
\bar B(D_i,C_\alpha):=\sum_{y\in C_\alpha}B(x,y)\quad(x\in D_i).
\]
Then \(\bar A,\bar B\) are row-stochastic, \(Q=AB\) lumps by \(\mathcal C\),
\(K=BA\) lumps by \(\mathcal D\), and
\[
\bar Q=\bar A\bar B,
\qquad
\bar K=\bar B\bar A.
\]
Hence
\[
\Spec_{\ne0}(\bar Q)=\Spec_{\ne0}(\bar K)
\]
with matching algebraic multiplicities.
\end{theorem}

\begin{proof}
Row-stochasticity follows by summing over all target blocks. For \(y\in C_\alpha\),
\[
\sum_{y'\in C_\beta}Q(y,y')
=
\sum_i\sum_{x\in D_i}A(y,x)\sum_{y'\in C_\beta}B(x,y')
=
\sum_i\bar A(C_\alpha,D_i)\bar B(D_i,C_\beta),
\]
where the first equality uses compatibility of \(B\), and the second uses
compatibility of \(A\).
Thus \(Q\) strongly lumps and \(\bar Q=\bar A\bar B\). The proof for \(K\) is
identical:
\[
\sum_{x'\in D_j}K(x,x')
=
\sum_\alpha\bar B(D_i,C_\alpha)\bar A(C_\alpha,D_j).
\]
The spectral equality follows by the same \(AB/BA\) argument used in
Theorem~\ref{thm:shared-spectrum}.
\end{proof}

Specializing the paired quotient theorem to the Burnside legs gives explicit
block-sum formulas.

\begin{corollary}[Burnside paired quotient formulas]
\label{cor:burnside-paired-quotient-formulas}
Let \(\mathcal C=\{C_\alpha\}\) be a partition of \(G^*\), and let \(\mathcal D=\{D_i\}\) be a partition of \(X\). Assume
\((\mathcal C,\mathcal D)\) is compatible for the Burnside legs. Then
\[
\bar A(C_\alpha,D_i)=\frac{|X_g\cap D_i|}{|X_g|}
\quad(g\in C_\alpha),
\qquad
\bar B(D_i,C_\alpha)=\frac{|G_x\cap C_\alpha|}{|G_x|}
\quad(x\in D_i).
\]
Moreover, the lumped kernels are
\[
\bar Q(C_\alpha,C_\beta)
=
\sum_{h\in C_\beta}Q(g,h)
=
\frac1{|X_g|}
\sum_{x\in X_g}\frac{|G_x\cap C_\beta|}{|G_x|}
\qquad(g\in C_\alpha),
\]
and
\[
\bar K(D_i,D_j)
=
\sum_{y\in D_j}K(x,y)
=
\frac1{|G_x|}
\sum_{a\in G_x}\frac{|X_a\cap D_j|}{|X_a|}
\qquad(x\in D_i).
\]
The stationary laws are
\[
\bar\pi_Q(C_\alpha)=\frac1Z\sum_{g\in C_\alpha}|X_g|,
\qquad
\bar\pi_K(D_i)=\frac1Z\sum_{x\in D_i}|G_x|,
\qquad
Z=|G|\,|X/G|.
\]
\end{corollary}

\begin{proof}
The formulas for \(\bar A\) and \(\bar B\) follow from
Lemma~\ref{lem:subset-agg} with \(\Omega=D_i\) and \(S=C_\alpha\). By
Theorem~\ref{thm:paired-lumpings-general}, \(Q\) and \(K\) lump by
\(\mathcal C\) and \(\mathcal D\). Hence
\[
\bar Q(C_\alpha,C_\beta)=\sum_{h\in C_\beta}Q(g,h),
\qquad
\bar K(D_i,D_j)=\sum_{y\in D_j}K(x,y).
\]
The two block-sum formulas again follow from Lemma~\ref{lem:subset-agg}, now with \(S=C_\beta\) and \(\Omega=D_j\). The stationary laws are the pushforwards of
\(\pi_Q(g)=|X_g|/Z\) and \(\pi_K(x)=|G_x|/Z\).
\end{proof}

The useful paired quotients are:
\[
\begin{array}{c|c|c|c}
\text{pair} & \text{partition of }X & \text{partition of }G^* & \text{lumpability}\\
\hline
\text{orbit/conjugacy} & X/G & \Conj(G)\cap G^* & \text{model-free}\\[2pt]
\text{exact stabilizer/fixed set} & D_H=\{x:G_x=H\} & E_F=\{g:X_g=F\} & \text{model-free}\\[2pt]
\text{support/fixed-symbol count} & P_m=\{x:|\supp(x)|=m\} & C_r=\{g:|\Fix(g)|=r\} & \text{value model}
\end{array}
\]
In the last row, for the value-permutation model \(S_k\curvearrowright[k]^n\),
\[
\supp(x):=\{x_1,\dots,x_n\},
\qquad
\Fix(g):=\{a\in[k]:g(a)=a\}.
\]
This last quotient pair is discussed in Section~\ref{subsec:fixed-set-and-count-quotients}.

\begin{remark}[Quotient notation]\label{rem:quotient-notation}
Superscripts specify the quotient pair when needed; for example,
\(\bar Q^{\mathrm{oc}}\) is the conjugacy-class quotient of \(Q\) paired with
the orbit quotient \(\bar K^{\mathrm{oc}}\). We omit the superscript when the
quotient is clear from context.
\end{remark}

\begin{corollary}[Orbit/conjugacy quotient pair]
\label{prop:orbit-conj-paired}
Let \(\mathcal O=X/G\), and let \(\mathcal C\) be the conjugacy classes in
\(G^*\). Then the pair \((\mathcal C,\mathcal O)\) is compatible. Hence, with
\[
\bar A^{\mathrm{oc}}([g],[x])=\frac{|X_g\cap[x]|}{|X_g|},
\qquad
\bar B^{\mathrm{oc}}([x],[g])=\frac{|G_x\cap[g]|}{|G_x|},
\]
we have
\[
\bar Q^{\mathrm{oc}}=\bar A^{\mathrm{oc}}\bar B^{\mathrm{oc}},
\qquad
\bar K^{\mathrm{oc}}=\bar B^{\mathrm{oc}}\bar A^{\mathrm{oc}},
\]
and
\[
\Spec_{\ne0}(\bar Q^{\mathrm{oc}})
=
\Spec_{\ne0}(\bar K^{\mathrm{oc}})
\]
with matching algebraic multiplicities.
\end{corollary}

\begin{proof}
Compatibility follows from Lemma~\ref{lem:inv-block}. The formulas for
\(\bar A^{\mathrm{oc}}\) and \(\bar B^{\mathrm{oc}}\) are Corollary~\ref{cor:burnside-paired-quotient-formulas}. The factorization and
spectral equality follow from Theorem~\ref{thm:paired-lumpings-general}.
\end{proof}

\begin{remark}
This quotient pair is not generally spectral-complete for \(Q,K\). The quotient TV distances give lower bounds for the original chains, with guaranteed equality in the cases covered by Corollaries~\ref{cor:lump-K-TV} and
\ref{cor:lump-Q-TV}.
\end{remark}

\begin{theorem}[Exact stabilizer/fixed-set quotient pair]
\label{thm:exact-fix-stab-pair}\label{thm:model-free-fixed-set-quotient}
Let
\[
\mathfrak F:=\{F\subseteq X:F=X_g\text{ for some }g\in G^*\},
\qquad
E_F:=\{g\in G^*:X_g=F\},
\]
and
\[
\mathfrak S:=\{H\le G:H=G_x\text{ for some }x\in X\},
\qquad
D_H:=\{x\in X:G_x=H\}.
\]
Then \(\{E_F\}\) and \(\{D_H\}\) are compatible. The quotient legs are
\[
\bar A^{\mathrm{fs}}(F,H)=\frac{|F\cap D_H|}{|F|},
\qquad
\bar B^{\mathrm{fs}}(H,F)=\frac{|H\cap E_F|}{|H|},
\]
and
\[
\bar Q^{\mathrm{fs}}=\bar A^{\mathrm{fs}}\bar B^{\mathrm{fs}},
\qquad
\bar K^{\mathrm{fs}}=\bar B^{\mathrm{fs}}\bar A^{\mathrm{fs}}.
\]
Moreover,
\[
\bar Q^{\mathrm{fs}}(F,F')
=
\frac{|E_{F'}|}{|F|}
\sum_{u\in F\cap F'}\frac1{|G_u|},
\qquad
\bar K^{\mathrm{fs}}(H,L)
=
\frac{|D_L|}{|H|}
\sum_{a\in H\cap L}\frac1{|X_a|}.
\]
Their stationary laws are
\[
\bar\pi_Q^{\mathrm{fs}}(F)=\frac{|E_F|\,|F|}{Z},
\qquad
\bar\pi_K^{\mathrm{fs}}(H)=\frac{|D_H|\,|H|}{Z},
\qquad
Z=|G|\,|X/G|.
\]
Finally,
\[
\Spec_{\ne0}(Q)=\Spec_{\ne0}(\bar Q^{\mathrm{fs}})
=\Spec_{\ne0}(K)=\Spec_{\ne0}(\bar K^{\mathrm{fs}}),
\]
with matching algebraic multiplicities, equivalently
\[
\det(\lambda I_{G^*}-Q)
=
\lambda^{|G^*|-|\mathfrak F|}
\det(\lambda I_{\mathfrak F}-\bar Q^{\mathrm{fs}}),
\qquad
\det(\lambda I_X-K)
=
\lambda^{|X|-|\mathfrak S|}
\det(\lambda I_{\mathfrak S}-\bar K^{\mathrm{fs}}).
\]
For \(g\in E_F\), \(x\in D_H\), and \(t\ge1\),
\[
d_Q(g,t)=d_{\bar Q^{\mathrm{fs}}}(F,t),
\qquad
d_K(x,t)=d_{\bar K^{\mathrm{fs}}}(H,t).
\]
\end{theorem}

\begin{proof}
For any \(g\in E_F\) and \(x\in D_H\),
\[
\frac{|X_g\cap D_H|}{|X_g|}
=
\frac{|F\cap D_H|}{|F|},
\qquad
\frac{|G_x\cap E_F|}{|G_x|}
=
\frac{|H\cap E_F|}{|H|}.
\]
Hence the pair \((\{E_F\},\{D_H\})\) is compatible. Corollary~\ref{cor:burnside-paired-quotient-formulas} gives the formulas for \(\bar A^{\mathrm{fs}}\) and \(\bar B^{\mathrm{fs}}\) and the stationary laws, and Theorem~\ref{thm:paired-lumpings-general} gives
\[
\bar Q^{\mathrm{fs}}=\bar A^{\mathrm{fs}}\bar B^{\mathrm{fs}},
\qquad
\bar K^{\mathrm{fs}}=\bar B^{\mathrm{fs}}\bar A^{\mathrm{fs}}.
\]
By Corollary~\ref{cor:burnside-paired-quotient-formulas}, together with
\[
|G_u\cap E_{F'}|=|E_{F'}|\mathbf 1_{\{u\in F'\}},
\qquad
|X_a\cap D_L|=|D_L|\mathbf 1_{\{a\in L\}},
\]
we get the formulas for \(\bar Q^{\mathrm{fs}}(F,F')\) and \(\bar K^{\mathrm{fs}}(H,L)\).

Define
\[
R_Q(g,F):=\mathbf 1_{\{g\in E_F\}},
\qquad
L_Q(F,g):=\frac{\mathbf 1_{\{g\in E_F\}}}{|E_F|}.
\]
Then we have
\[
L_QR_Q=I_{\mathfrak F},
\qquad
Q=R_Q\bar Q^{\mathrm{fs}}L_Q,
\qquad
\bar Q^{\mathrm{fs}}=L_QQR_Q.
\]
Thus \(Q=(R_Q\bar Q^{\mathrm{fs}})L_Q\) and \(L_Q(R_Q\bar Q^{\mathrm{fs}})=\bar Q^{\mathrm{fs}}\). The same \(AB/BA\) matrix fact (used in Theorem~\ref{thm:shared-spectrum}) and Sylvester's identity give \(\Spec_{\ne0}(Q)=\Spec_{\ne0}(\bar Q^{\mathrm{fs}})\) with matching algebraic multiplicities and the determinant identity for \(Q\).
The proof for \(K\) is identical using \(D_H\). Together with Theorem~\ref{thm:paired-lumpings-general}, this gives the common nonzero spectrum.

Finally, for \(t\ge1\), \(Q^t(g,\cdot)\) and \(\pi_Q\) are constant on each \(E_F\), and \(K^t(x,\cdot)\) and \(\pi_K\) are constant on each \(D_H\).
Therefore Proposition~\ref{prop:TV-lumping}\textup{(b)} gives the two TV
identities.
\end{proof}

Finally, we show that coarsening conjugacy classes by the fixed-state count \(|X_g|\) need not give a lumping. For an action \(G\curvearrowright X\), set
\[
B_s:=\{g\in G^*:|X_g|=s\}.
\]
Since \(|X_g|\) is a class function, i.e. constant on conjugacy classes, the
partition \(\{B_s\}\) is coarser than the conjugacy-class partition. Strong lumpability would require
\[
\sum_{h\in B_s}Q(g,h)
\]
to depend only on \(|X_g|\). The following example shows this can fail. Even when a
coarser lumping exists, TV only contracts under pushforward unless the
block-constancy condition in Proposition~\ref{prop:TV-lumping} holds.

\begin{example}[Fixed-state-count lumping can fail]\label{ex:cyclecount-fail-n4}
In the coordinate-permutation model \(S_4\curvearrowright\{0,1\}^4\), \(|X_\sigma|=2^{c(\sigma)}\), where \(c(\sigma)\) is the number of cycles of \(\sigma\). Take
\[
g_1=(12)(34),
\qquad
g_2=(123).
\]
Then \(|X_{g_1}|=|X_{g_2}|=4\). For
\[
B_4=\{h\in S_4:|X_h|=4\}=\{h\in S_4:c(h)=2\},
\]
a direct calculation gives
\[
\sum_{h\in B_4}Q(g_1,h)=\frac{17}{48},
\qquad
\sum_{h\in B_4}Q(g_2,h)=\frac{19}{48}.
\]
Thus fixed-state-count lumping is not strongly lumpable for \(Q\).
\end{example}

\subsection{Common Duals Under Uniform Covers}

\begin{definition}[Uniform stabilizer-preserving cover]\label{def:uniform-cover}
Let \(G\) act on finite sets \(X\) and \(Y\). A map \(p:X\to Y\) is a uniform
stabilizer-preserving \(G\)-cover of degree \(m\ge1\) if
\begin{enumerate}[label=\textup{(\roman*)},itemsep=2pt]
\item \(p(gx)=gp(x)\) for all \(g\in G\), \(x\in X\);
\item \(|p^{-1}(y)|=m\) for all \(y\in Y\);
\item \(G_x=G_{p(x)}\) for all \(x\in X\).
\end{enumerate}
\end{definition}

\begin{proposition}[Free commuting slice criterion]
\label{prop:free-commuting-slice}
Let \(G\) and \(H\) act on a finite set \(X\), and let \(Y\subseteq X\).
Assume the actions commute (\(g(hx)=h(gx)\)), the \(H\)-action is free (\(hx=x\Rightarrow h=e_H\)), \(Y\) is \(G\)-invariant (\(x\in Y\Rightarrow gx\in Y\)), and \(Y\) is a slice for the \(H\)-orbits (\(|Y\cap Hx|=1\) for \(x\in X\)).
Let \(p:X\to Y\) send \(x\) to the unique point of \(Y\cap Hx\). Then \(p\) is a uniform stabilizer-preserving \(G\)-cover of degree \(|H|\).
\end{proposition}

\begin{proof}
The slice condition makes \(p\) well defined, and \(p(y)=y\) for \(y\in Y\), so
\(p\) is onto. For \(y\in Y\),
\[
p^{-1}(y)=Hy,
\]
because \(p(x)=y\iff y\in Hx\iff x\in Hy\). Since the \(H\)-action is free,
\[
|p^{-1}(y)|=|Hy|=|H|.
\]
Let \(p(x)=y\), and choose \(h\in H\) with \(x=hy\). For \(g\in G\),
\[
gx=g(hy)=h(gy).
\]
By \(G\)-invariance, \(gy\in Y\), so the unique point of \(Y\cap H(gx)\) is
\(gy\). Hence
\[
p(gx)=gy=gp(x).
\]
Finally,
\[
gx=x\iff g(hy)=hy\iff h(gy)=hy\iff gy=y,
\]
so \(G_x=G_y=G_{p(x)}\). Thus \(p\) is a uniform stabilizer-preserving
\(G\)-cover of degree \(|H|\).
\end{proof}

\begin{theorem}[Common dual under a uniform cover]\label{thm:common-dual-cover}
Let \(p:X\to Y\) be a uniform stabilizer-preserving \(G\)-cover of degree \(m\).
Then:
\begin{enumerate}[label=\textup{(\alph*)},itemsep=2pt]
\item \(X_g=p^{-1}(Y_g)\) and \(|X_g|=m|Y_g|\) for every \(g\in G\). Hence \(G_X^*=G_Y^*\).
\item On this common dual state space, \(Q_X=Q_Y\).
\item \(K_X(x,x')=m^{-1}K_Y(p(x),p(x'))\). Thus \(K_X\) strongly lumps by the fibers of \(p\) to \(K_Y\). Moreover, for the orbit-lumped kernels,
\[
\bar K_X([x],[x'])
=
\frac1m\bar K_Y([p(x)],[p(x')]).
\]
\item For \(t\ge1\),
\[
\|K_X^t(x,\cdot)-\pi_{K_X}\|_{\TV}
=
\|K_Y^t(p(x),\cdot)-\pi_{K_Y}\|_{\TV},
\qquad
d_{K_X}(t)=d_{K_Y}(t).
\]
\item \(\pi_{Q_X}=\pi_{Q_Y}\), and \(\pi_{K_X}(x)=m^{-1}\pi_{K_Y}(p(x))\).
\item \(\Spec_{\ne0}(K_X)=\Spec_{\ne0}(Q_X)=\Spec_{\ne0}(Q_Y)=\Spec_{\ne0}(K_Y)\), with matching algebraic multiplicities.
\end{enumerate}
\end{theorem}

\begin{proof}
If \(x\in X_g\), then \(gp(x)=p(gx)=p(x)\), so \(p(x)\in Y_g\). Conversely,
if \(p(x)\in Y_g\), then \(g\in G_{p(x)}=G_x\), so \(x\in X_g\). Hence
\(X_g=p^{-1}(Y_g)\), and \(|X_g|=m|Y_g|\). This gives \(G_X^*=G_Y^*\).

For \(g,h\) in the common dual state space,
\[
\begin{aligned}
Q_X(g,h)
&=\frac1{m|Y_g|}\sum_{y\in Y_g\cap Y_h}
  \sum_{x\in p^{-1}(y)}\frac1{|G_x|} \\
&=\frac1{m|Y_g|}\sum_{y\in Y_g\cap Y_h}
  \sum_{x\in p^{-1}(y)}\frac1{|G_y|} \\
&=\frac1{|Y_g|}\sum_{y\in Y_g\cap Y_h}\frac1{|G_y|}
=Q_Y(g,h).
\end{aligned}
\]
where the second equality uses \(p(x)=y\), hence \(G_x=G_{p(x)}=G_y\). Similarly,
\[
K_X(x,x')
=\frac1{|G_{p(x)}|}\sum_{g\in G_{p(x)}\cap G_{p(x')}}\frac1{m|Y_g|}
=\frac1mK_Y(p(x),p(x')).
\]
Since \(p\) is equivariant and \(G_{x'}=G_{p(x')}\), the restriction
\(p|_{[x']}:[x']\to[p(x')]\) is a bijection. Summing the displayed identity over \(u\in[x']\) gives
\[
\bar K_X([x],[x'])
=
\frac1m\bar K_Y([p(x)],[p(x')]).
\]

By induction,
\[
K_X^t(x,u)=\frac1m K_Y^t(p(x),p(u))\qquad(t\ge1).
\]
Also, Burnside's lemma and \(|X_g|=m|Y_g|\) give
\[
|X/G|=\frac1{|G|}\sum_g |X_g|
=m\,\frac1{|G|}\sum_g |Y_g|
=m|Y/G|.
\]
Together with \(G_x=G_{p(x)}\), this gives
\[
\pi_{K_X}(x)=\frac1m\pi_{K_Y}(p(x)),
\qquad
\pi_{Q_X}=\pi_{Q_Y}.
\]
Summing over fibers gives
\[
\|K_X^t(x,\cdot)-\pi_{K_X}\|_{\TV}
=
\|K_Y^t(p(x),\cdot)-\pi_{K_Y}\|_{\TV}.
\]
Taking maxima over \(x\), and using that \(p\) is surjective, gives
\(d_{K_X}(t)=d_{K_Y}(t)\). Finally, \(Q_X=Q_Y\) and
Theorem~\ref{thm:shared-spectrum} give the spectral statement.
\end{proof}

\begin{corollary}[Fixed-point proportionality obstruction]\label{cor:common-dual-obstruction}
Assume \(G_X^*=G_Y^*\) and \(Q_X=Q_Y\). Then the fixed-point profiles are globally proportional: there is a single constant \(c>0\), independent of \(g\), such that
\[
|X_g|=c|Y_g|\qquad(g\in G_X^*).
\]
\end{corollary}

\begin{proof}
Since \(Q_X=Q_Y\) on the common state space and the dual chains are irreducible,
their stationary laws are equal. By Theorem~\ref{thm:universal-pi},
\[
\frac{|X_g|}{|G|\,|X/G|}
=
\frac{|Y_g|}{|G|\,|Y/G|}
\qquad(g\in G_X^*).
\]
Hence
\[
|X_g|=\frac{|X/G|}{|Y/G|}\,|Y_g|,
\]
with a constant independent of \(g\).
\end{proof}

\begin{corollary}[Equivariant bijection transfer]
\label{cor:equivariant-bijection-transfer}
Let \(G\) act on finite sets \(X\) and \(Y\). Suppose \(\phi:X\to Y\) is a \(G\)-equivariant bijection, i.e.,
\[
\phi(gx)=g\phi(x)
\qquad(g\in G,\ x\in X).
\]
Then, after identifying \(x\) with \(\phi(x)\),
\[
K_X=K_Y,
\qquad
Q_X=Q_Y.
\]
Hence the primal chains have the same stationary laws, spectra, and total variation distances, and the dual chains are identical.
\end{corollary}

\begin{proof}
For \(g\in G\),
\[
g\in G_x
\iff gx=x
\iff \phi(gx)=\phi(x)
\iff g\phi(x)=\phi(x)
\iff g\in G_{\phi(x)}.
\]
Thus \(\phi\) is a uniform stabilizer-preserving \(G\)-cover of degree \(1\).
The result follows from Theorem~\ref{thm:common-dual-cover}.
\end{proof}

\begin{example}[Parking functions and the Bose-Einstein dual]\label{ex:parking-BE-common-dual}
Let \(k=n+1\), identify \([k]\) with \(\mathbb Z/k\mathbb Z\), and view the classical parking functions \(\PF_n\) as a subset of \([k]^n\). Let \(S_n\) act by permuting coordinates, and let \(H=\mathbb Z/k\mathbb Z\) act by global shifts,
\[
a\cdot(x_1,\dots,x_n)=(x_1+a,\dots,x_n+a)\pmod{k}.
\]
The \(H\)-action is free, it commutes with the \(S_n\)-action, and \(\PF_n\) is
\(S_n\)-invariant. By Pollak's circular parking argument, \(\PF_n\) meets each
\(H\)-orbit in exactly one point~\cite{Stanley97PF}. Hence
Proposition~\ref{prop:free-commuting-slice} gives a uniform
stabilizer-preserving \(S_n\)-cover \([k]^n\to\PF_n\) of degree \(k\). Thus
\[
Q_{\PF_n}=Q_{[k]^n}.
\]
Also \(|(\PF_n)_\sigma|=k^{c(\sigma)-1}\), and the primal Burnside distances
agree for \(t\ge1\). Aldous's bound gives
\[
d_{K_{\PF_n}}(t)
\le n\left(1-\frac1{n+1}\right)^t
\qquad(t\ge1).
\]
Thus the dual chain identifies the parking-function Burnside process with the
Bose--Einstein dual at \(k=n+1\), giving the same dual transition matrix and the
same nonzero spectral data, although the primal state spaces are different.
\end{example}

\begin{example}[Zero-sum words]
\label{ex:zero-sum-common-dual}
Let \(m\ge2\), \(n\ge1\), and \(\gcd(m,n)=1\). Identify \([m]\) with \(\mathbb Z/m\mathbb Z\). Let \(G\le S_n\) act on \(X:=(\mathbb Z/m\mathbb Z)^n\) by permuting coordinates, and let \(H=\mathbb Z/m\mathbb Z\) act by global shifts as above. Set
\[
Y:=\left\{x\in X:\sum_{i=1}^n x_i=0\right\}.
\]
The \(H\)-action is free, commutes with the \(G\)-action, and \(Y\) is \(G\)-invariant. Each \(H\)-orbit meets \(Y\) exactly once: for \(x\in X\), the required shift \(c\) satisfies
\[
\sum_i(x_i+c)=\sum_i x_i+nc=0,
\]
and \(n\) is invertible modulo \(m\).
Hence Proposition~\ref{prop:free-commuting-slice} gives a uniform
stabilizer-preserving \(G\)-cover \(X\to Y\) of degree \(m\). Therefore
Theorem~\ref{thm:common-dual-cover} gives
\[
Q_X=Q_Y.
\]

If \(G=S_n\) and \(c(\sigma)\) is the number of cycles of \(\sigma\), then
\[
|X_\sigma|=m^{c(\sigma)}.
\]
Since the cover has degree \(m\), Theorem~\ref{thm:common-dual-cover} gives
\[
|Y_\sigma|=m^{c(\sigma)-1}.
\]
\end{example}

\section{The Value-Permutation Model}\label{sec:value}

This section studies \(S_k\curvearrowright[k]^n\), where \(S_k\) relabels
symbols. For \(k\ge2\), the model-free fixed-set quotient becomes especially small here: \(X_g=\Fix(g)^n\), so the quotient is indexed by fixed-symbol sets and has size \(2^k-k-1\), independent of \(n\).
We also give closed forms, a coarser
fixed-point-count quotient, and the mixing time bounds.

\subsection{Setting and Basic Properties}

Let \(G=S_k\) act on \(X=[k]^n\) by
\[
(g\cdot x)_i=g(x_i).
\]
For \(g\in S_k\), write
\[
\Fix(g):=\{a\in[k]:g(a)=a\},
\qquad
f(g):=|\Fix(g)|.
\]
For \(x\in[k]^n\), write
\[
\supp(x):=\{x_1,\dots,x_n\},
\qquad
r_x:=|\supp(x)|.
\]

\begin{lemma}[Basic identities]\label{lem:value-basic}
For \(g\in S_k\) and \(x\in[k]^n\),
\[
X_g=\Fix(g)^n,
\qquad
|X_g|=f(g)^n,
\]
\[
G_x=\{\sigma\in S_k:\sigma(a)=a\ \forall a\in\supp(x)\}\cong S_{k-r_x},
\qquad
|G_x|=(k-r_x)!.
\]
Also,
\begin{equation}\label{eq:value-orbit-count}
|[k]^n/S_k|=\sum_{r=0}^{\min\{k,n\}}S(n,r),
\end{equation}
where \(S(n,r)\) is the Stirling number of the second kind.
For \(g,h\in S_k^*\), the dual kernel is
\[
Q(g,h)=\frac1{f(g)^n}\sum_{x\in X_g\cap X_h}\frac1{(k-r_x)!}.
\]
\end{lemma}

\begin{proof}
The condition \(g\cdot x=x\) is equivalent to \(x_i\in\Fix(g)\) for all \(i\).
The condition \(\sigma\cdot x=x\) is equivalent to fixing every symbol in
\(\supp(x)\) pointwise and permuting the \(k-r_x\) unused symbols freely. Finally, two words are
in the same orbit exactly when they induce the same set partition of \([n]\)
into value-classes.
\end{proof}

Write \(!m\) for the number of derangements of \(m\) symbols, with \(!0=1\). The dual state space is
\[
S_k^*=\{g\in S_k:f(g)>0\}=S_k\setminus\{\text{derangements}\},
\]
with size \(k!-!\,k\). As \(k\to\infty\), the fraction of derangements in \(S_k\) tends to \(e^{-1}\).
Thus \(S_k^*\) still has asymptotic density \(1-e^{-1}\) in \(S_k\).

\subsection{Fixed-Symbol-Set Closed Forms}\label{sec:value-canonical}

For \(g,h\in S_k^*\), set
\[
a:=f(g),
\qquad
j:=|\Fix(g)\cap\Fix(h)|.
\]
Then
\[
X_g\cap X_h
=
(\Fix(g)\cap\Fix(h))^n,
\]
so \(|X_g\cap X_h|=j^n\).

\begin{lemma}[Words with exactly \(r\) symbols]\label{lem:stirl}
For \(1\le r\le\min\{j,n\}\),
\[
|\{x\in[j]^n:r_x=r\}|=\binom jrS(n,r)r!.
\]
\end{lemma}

\begin{proof}
Choose the \(r\) symbols used by the word: \(\binom jr\) choices. Partition the
\(n\) positions into \(r\) nonempty blocks: \(S(n,r)\) choices. Assign the
chosen symbols bijectively to the blocks: \(r!\) choices. Multiplying gives
\[
|\{x\in[j]^n:r_x=r\}|=\binom jr S(n,r)r!. \qedhere
\]
\end{proof}

\begin{theorem}[Canonical closed forms for \(Q(g,h)\)]\label{thm:value-forms}
We have
\begin{align}
Q(g,h)
&=\frac1{a^n}\sum_{r=1}^{\min\{j,n\}}
\binom jrS(n,r)\frac{r!}{(k-r)!},\label{eq:value-stirling}\\
&=\left(\frac ja\right)^n
\E[(k-R_{j,n})!^{-1}],\label{eq:value-expectation}\\
&=\frac{n!}{a^n}[u^n][z^k]e^z(1+z(e^u-1))^j.\label{eq:value-coeff}
\end{align}
Here, for \(j\ge1\), \(R_{j,n}\) is the number of distinct symbols used by a
uniform word in \([j]^n\):
\[
\Prob(R_{j,n}=r)=\frac{\binom jrS(n,r)r!}{j^n}.
\]
If \(j=0\), all formulas are interpreted as \(0\).
\end{theorem}

\begin{proof}
If \(j=0\), then \(X_g\cap X_h=\varnothing\). Assume \(j\ge1\). Then
\[
Q(g,h)=\frac1{a^n}\sum_{x\in[j]^n}\frac1{(k-r_x)!}.
\]
Apply Lemma~\ref{lemma:reindexing} to
\[
\Phi:[j]^n\to\{1,\dots,\min\{j,n\}\},
\qquad
\Phi(x)=r_x.
\]
By Lemma~\ref{lem:stirl}, \(|\Phi^{-1}(r)|=\binom jrS(n,r)r!\), giving
\eqref{eq:value-stirling}. Since
\[
\binom jrS(n,r)r!=j^n\Prob(R_{j,n}=r),
\]
\eqref{eq:value-stirling} becomes
\[
Q(g,h)
=
\frac{j^n}{a^n}\sum_r \Prob(R_{j,n}=r)(k-r)!^{-1}
=
\left(\frac ja\right)^n
\E[(k-R_{j,n})!^{-1}],
\]
which is \eqref{eq:value-expectation}. For \eqref{eq:value-coeff}, use
\[
r!S(n,r)=n![u^n](e^u-1)^r,
\qquad
\frac1{(k-r)!}=[z^k]z^re^z,
\]
and the binomial theorem.
\end{proof}

\begin{example}[Identity start]
If \(g=e\) and \(f:=f(h)\), then \(a=k\), \(j=f\). If \(f=1\),
\[
Q(e,h)=\frac1{k^n(k-1)!}.
\]
If \(f=2\),
\[
Q(e,h)=\frac{2}{k^n}\left(\frac1{(k-1)!}+\frac{2^{n-1}-1}{(k-2)!}\right).
\]
\end{example}

\subsection{Stationary Distribution}\label{subsec:sta-value}

\begin{theorem}[Dual stationary law in the value model]\label{thm:stationary}
For \(S_k\curvearrowright[k]^n\),
\[
\pi_Q(g)=\frac{f(g)^n}{k!Z_{k,n}},
\qquad
Z_{k,n}:=\sum_{r=0}^{\min\{k,n\}}S(n,r).
\]
If \(k\ge n\), then \(Z_{k,n}=B_n\), the \(n\)th Bell number.
\end{theorem}

\begin{proof}
Use Theorem~\ref{thm:universal-pi} and Lemma~\ref{lem:value-basic}.
\end{proof}

On the other hand, for the primal chain,
\[
\pi_K(x)=\frac{(k-r_x)!}{k!Z_{k,n}}.
\]

\begin{lemma}[Extrema of the stationary laws]\label{lem:val-ext}
For \(n\ge1\),
\[
\pi_{K,\max}=\frac1{kZ_{k,n}},
\qquad
\pi_{K,\min}=\frac{(k-r_{\max})!}{k!Z_{k,n}},
\qquad
r_{\max}:=\min\{k,n\}.
\]
For \(n\ge1\) and \(k\ge3\),
\[
\pi_{Q,\min}=\frac1{k!Z_{k,n}},
\qquad
\pi_{Q,\max}=\frac{k^n}{k!Z_{k,n}},
\qquad
\frac{\pi_{Q,\max}}{\pi_{Q,\min}}=k^n.
\]
For \(k=1,2\), \(S_k^*=\{e\}\).
\end{lemma}

\begin{proof}
Since \(\pi_K(x)\propto(k-r_x)!\), the maximum occurs at \(r_x=1\) and the
minimum at \(r_x=r_{\max}\). Since \(\pi_Q(g)\propto f(g)^n\), for \(k\ge3\)
the minimum positive value is \(f(g)=1\), and the maximum is \(f(e)=k\).
\end{proof}

\begin{remark}[Cycle-index check]
For \(\sigma\in S_k\), let \(c_j(\sigma)\) denote the number of \(j\)-cycles of
\(\sigma\), and define the cycle-index polynomial
\[
P_k(x_1,\dots,x_k):=
\frac{1}{k!}\sum_{\sigma\in S_k}\prod_{j=1}^k x_j^{c_j(\sigma)}.
\]
For cycle-index methods beyond symmetric groups, see Fulman~\cite{FulmanCycleIndex}. The cycle-index identity \cite[page~412]{Diaconis05}
\[
\sum_{k\ge0}P_k(x_1,\dots,x_k)z^k=
\exp\left(\sum_{j\ge1}\frac{x_jz^j}{j}\right)
\]
gives
\[
F_k(x):=\sum_{g\in S_k}x^{f(g)}=k![z^k]\frac{e^{(x-1)z}}{1-z}
=k!\sum_{m=0}^k\frac{(x-1)^m}{m!}.
\]
Hence
\[
\sum_{g\in S_k}f(g)^n=\left.(x\tfrac{d}{dx})^nF_k(x)\right|_{x=1}
=k!\sum_{m=0}^kS(n,m),
\]
which matches the normalization above. Also, for uniform \(g\in S_k\) and
\(k\ge2\), \(\E f(g)=1\) and \(\Var(f(g))=1\).
\end{remark}

\subsection{Fixed-Set and Fixed-Point-Count Quotients}\label{subsec:fixed-set-and-count-quotients}

The exact quotient remembers \(\Fix(g)\subseteq[k]\); it preserves the entire
nonzero spectrum. The coarser quotient remembers only \(f(g)=|\Fix(g)|\); it is
smaller but may lose eigenvalues. Assume \(k\ge2\), and write
\[
I_k:=\{1,2,\ldots,k\}\setminus\{k-1\}.
\]

\subsubsection{The exact fixed-symbol-set quotient}

The fixed-set quotient of Theorem~\ref{thm:exact-fix-stab-pair} is explicit here. Since \(X_g=\Fix(g)^n\), fixed sets are indexed by fixed-symbol sets. For
\(S\subseteq[k]\), set
\[
E_S:=\{g\in S_k^*:\Fix(g)=S\}.
\]
Then \(E_S\ne\varnothing\) iff \(S\ne\varnothing\) and \(|S|\ne k-1\), and in that case
\[
|E_S|=!\,(k-|S|).
\]
Let
\[
\mathcal B_k:=\{S\subseteq[k]:S\ne\varnothing,\ |S|\ne k-1\}.
\]
For \(k\ge2\),
\[
|\mathcal B_k|=2^k-k-1.
\]

\begin{corollary}[Exact fixed-symbol-set quotient]
\label{thm:exact-fixed-set-quotient}
For \(S,T\in\mathcal B_k\), set \(j:=|S\cap T|\). Then
\begin{align}
\bar Q_n^{\mathrm{fs}}(S,T)
&=
!\,(k-|T|)\frac1{|S|^n}
\sum_{m=1}^{\min\{j,n\}}
\binom jm S(n,m)\frac{m!}{(k-m)!},
\label{eq:fixed-set-q-single}\\
&=
!\,(k-|T|)
\left(\frac{j}{|S|}\right)^n
\E[(k-R_{j,n})!^{-1}],
\label{eq:fixed-set-q-expect}\\
&=
!\,(k-|T|)
\frac{n!}{|S|^n}
[u^n][z^k]e^z(1+z(e^u-1))^j.
\label{eq:fixed-set-q-coeff}
\end{align}
The second formula is interpreted as \(0\) when \(j=0\). The stationary law is
\[
\bar\pi_{Q,n}^{\mathrm{fs}}(S)=
\frac{!\,(k-|S|)|S|^n}{k!Z_{k,n}}.
\]
Moreover,
\[
\Spec_{\ne0}(K_n)=\Spec_{\ne0}(Q_n)
=
\Spec_{\ne0}(\bar Q_n^{\mathrm{fs}}),
\qquad
\rank(K_n)=\rank(Q_n)\le 2^k-k-1,
\]
with matching algebraic multiplicities, and
\[
\det(\lambda I_{S_k^*}-Q_n)
=
\lambda^{|S_k^*|-|\mathcal B_k|}
\det(\lambda I_{\mathcal B_k}-\bar Q_n^{\mathrm{fs}}).
\]
If \(g\in E_S\), then for \(t\ge1\),
\[
d_Q(g,t)=d_{\bar Q_n^{\mathrm{fs}}}(S,t),
\qquad
d_Q(t)=\max_{S\in\mathcal B_k}d_{\bar Q_n^{\mathrm{fs}}}(S,t).
\]
\end{corollary}

\begin{proof}
For \(g\in E_S\) and \(h\in E_T\), since \(|E_T|=!\,(k-|T|)\), Theorem~\ref{thm:exact-fix-stab-pair} gives
\[
\bar Q_n^{\mathrm{fs}}(S,T)=|E_T|Q_n(g,h).
\]
Substituting the three formulas for \(Q_n(g,h)\) from
Theorem~\ref{thm:value-forms} gives
\eqref{eq:fixed-set-q-single}--\eqref{eq:fixed-set-q-coeff}.

The stationary law is the pushforward of
\[
\pi_Q(g)=\frac{f(g)^n}{k!Z_{k,n}}.
\]
All spectral, determinant, rank, and TV claims are the corresponding conclusions
of Theorem~\ref{thm:exact-fix-stab-pair}.
\end{proof}

\begin{corollary}[Fixed-\(k\) asymptotic spectrum]
\label{cor:fixed-k-asymptotic-spectrum}
Fix \(k\ge3\). As \(n\to\infty\),
\[
\bar Q_n^{\mathrm{fs}}(S,T)
\longrightarrow
\bar Q_\infty^{\mathrm{fs}}(S,T):=
\frac{!\,(k-|T|)}{(k-|S|)!}\ind_{\{S\subseteq T\}}.
\]
Ordering \(\mathcal B_k\) so that \(S\) appears before \(T\) whenever \(S\subsetneq T\), the
limit is upper triangular with diagonal entries
\[
\theta_{k,|S|}:=\frac{!\,(k-|S|)}{(k-|S|)!}.
\]
Thus
\[
\Spec(\bar Q_n^{\mathrm{fs}})
\to
\left\{
\theta_{k,r}\text{ with multiplicity }\binom kr:
r\in\{1,\dots,k-2,k\}
\right\}.
\]
Since the limiting eigenvalues are nonzero, \(\bar Q_n^{\mathrm{fs}}\) is
nonsingular for all large \(n\). Consequently,
\[
\lambda_1(K_n)=\lambda_1(Q_n)\to\frac12,
\qquad
t_{\mathrm{rel}}(K_n)=t_{\mathrm{rel}}(Q_n)\to2.
\]
\end{corollary}

\begin{proof}
For \(g\in E_S\) and \(h\in E_T\), set \(q_n(S,T):=Q_n(g,h)\). By
Theorem~\ref{thm:value-forms},
\[
q_n(S,T)
=
\left(\frac{|S\cap T|}{|S|}\right)^n
\E[(k-R_{|S\cap T|,n})!^{-1}].
\]
If \(S\nsubseteq T\), then \(|S\cap T|<|S|\). Since
\((k-R_{|S\cap T|,n})!^{-1}\le1\), we get
\[
q_n(S,T)\le \left(\frac{|S\cap T|}{|S|}\right)^n\to0.
\]
If \(S\subseteq T\), then \(|S\cap T|=|S|\). Let \(r:=|S|\). A uniform word in
\([r]^n\) uses all \(r\) symbols with probability tending to \(1\), since
\[
\Prob(R_{r,n}<r)\le r\left(1-\frac1r\right)^n\to0.
\]
Thus \(R_{r,n}\to r\) in probability, and bounded convergence gives
\[
q_n(S,T)\to \frac1{(k-r)!}=\frac1{(k-|S|)!}.
\]
Multiplying by \(!\,(k-|T|)\) gives
\[
\bar Q_n^{\mathrm{fs}}(S,T)
\to
\frac{!\,(k-|T|)}{(k-|S|)!}\mathbf 1_{\{S\subseteq T\}}.
\]
The limiting matrix is upper triangular when \(\mathcal B_k\) is ordered by
inclusion, with diagonal entries \(\theta_{k,|S|}\).
Hence the eigenvalues converge to these diagonal entries. The unique trivial
limiting eigenvalue is \(\theta_{k,k}=1\). The largest nontrivial one is
\[
\theta_{k,k-2}=\frac{!2}{2!}=\frac12,
\]
and \(!m/m!<1/2\) for \(m\ge3\). Therefore
\[
\lambda_1(K_n)=\lambda_1(Q_n)\to\frac12,
\qquad
t_{\mathrm{rel}}(K_n)=t_{\mathrm{rel}}(Q_n)\to2
\]
by Corollary~\ref{cor:gap-equal}.
\end{proof}

\subsubsection{The coarser fixed-point-count quotient}

Assume \(k\ge2\). Recall \(I_k:=\{1,\dots,k\}\setminus\{k-1\}\). For
\(s\in I_k\), set
\[
C_s:=\{g\in S_k^*:f(g)=s\},
\qquad
\theta_{k,s}:=\frac{!\,(k-s)}{(k-s)!}.
\]

\begin{lemma}[Fixed-point counts]\label{lem:fixedpoint-counts}
For \(s\in I_k\):
\begin{enumerate}[label=\textup{(\alph*)},itemsep=2pt]
\item\[
|C_s|=\binom{k}{s}\cdot\,!\,(k-s).
\]
\item If \(f(g)=r\), then
\[
|\{h\in C_s:|\Fix(g)\cap\Fix(h)|=j\}|
=
\binom rj\binom{k-r}{s-j}\cdot\,!\,(k-s).
\]
\item If \(x\in[k]^n\), then
\[
G_x=\{h\in S_k:\supp(x)\subseteq\Fix(h)\},
\qquad
|G_x\cap C_s|=\binom{k-r_x}{s-r_x}\cdot\,!\,(k-s).
\]
\end{enumerate}
\end{lemma}

\begin{proof}
For \textup{(a)}, choose the \(s\) fixed symbols and derange the remaining
\(k-s\) symbols.

For \textup{(b)}, let \(F:=\Fix(g)\), \(|F|=r\). Choose \(j\) fixed symbols in
\(F\), choose the remaining \(s-j\) fixed symbols in \(F^c\), and derange the
remaining \(k-s\) symbols.

For \textup{(c)}, \(h\in G_x\) iff \(h\) fixes every symbol in \(\supp(x)\).
Thus the \(r_x\) symbols in \(\supp(x)\) are forced fixed; choose the remaining
\(s-r_x\) fixed symbols from the other \(k-r_x\) symbols, then derange the
remaining \(k-s\). We use the convention \(\binom ab=0\) when \(b\notin
\{0,\dots,a\}\).
\end{proof}

\begin{theorem}[Fixed-count/support-count quotient pair]
\label{thm:fixed-point-count-lumping}
Let
\[
P_m:=\{x\in[k]^n:r_x=m\},
\qquad
m\in\mathcal M_{k,n}:=\{1,\dots,\min\{k,n\}\}.
\]
The pair \(\{C_s\}\), \(\{P_m\}\) is compatible. Define
\[
\bar A^{\#}(r,m):=\frac{\binom rmS(n,m)m!}{r^n},
\qquad
\bar B^{\#}(m,s):=\frac{\binom{k-m}{s-m}\cdot\,!\,(k-s)}{(k-m)!}.
\]
Then
\[
\bar Q^{\#}=\bar A^{\#}\bar B^{\#},
\qquad
\bar K^{\#}=\bar B^{\#}\bar A^{\#},
\]
and
\[
\Spec_{\ne0}(\bar Q^{\#})=\Spec_{\ne0}(\bar K^{\#}),
\qquad
\Spec(\bar Q^{\#})\subseteq\Spec(Q_n),
\qquad
\Spec(\bar K^{\#})\subseteq\Spec(K_n).
\]

For \(r,s\in I_k\),
\begin{align}
\bar Q^{\#}(r,s)
&=
\theta_{k,s}\frac1{r^n}
\sum_{m=1}^{\min\{r,s,n\}}
\binom rmS(n,m)\frac{m!}{(s-m)!},
\label{eq:Qbar-from-Q-single}\\
&=
\theta_{k,s}\E[(s-R_{r,n})!^{-1}],
\label{eq:Qbar-from-Q-expect}\\
&=
\theta_{k,s}\frac{n!}{r^n}[u^n][z^s]e^z(1+z(e^u-1))^r.
\label{eq:Qbar-from-Q-coeff}
\end{align}
The stationary laws are
\[
\bar\pi_Q^{\#}(s)
=
\frac{\binom{k}{s}\cdot\,!\,(k-s)s^n}{k!Z_{k,n}}
=
\frac{\theta_{k,s}}{Z_{k,n}}\frac{s^n}{s!},
\qquad
\bar\pi_K^{\#}(m)=\frac{S(n,m)}{Z_{k,n}}.
\]
\end{theorem}

\begin{proof}
For \(g\in C_r\) and \(x\in P_m\), Lemmas~\ref{lem:stirl} and
\ref{lem:fixedpoint-counts} give
\[
\bar A^{\#}(r,m)=\frac{|X_g\cap P_m|}{|X_g|}
=\frac{\binom rmS(n,m)m!}{r^n},
\qquad
\bar B^{\#}(m,s)=\frac{|G_x\cap C_s|}{|G_x|}
=\frac{\binom{k-m}{s-m}\cdot\,!\,(k-s)}{(k-m)!}.
\]
Thus the pair is compatible. The factorization and nonzero spectral equality
follow from Theorem~\ref{thm:paired-lumpings-general}; the spectral inclusions
follow from Proposition~\ref{prop:lumping-facts}.

Since
\[
\bar B^{\#}(m,s)
=
\frac{!\,(k-s)}{(k-s)!}\frac1{(s-m)!}
=
\theta_{k,s}\frac1{(s-m)!},
\]
the identity \(\bar Q^{\#}=\bar A^{\#}\bar B^{\#}\) gives
\[
\bar Q^{\#}(r,s)
=
\theta_{k,s}\frac1{r^n}
\sum_{m=1}^{\min\{r,s,n\}}
\binom rmS(n,m)\frac{m!}{(s-m)!}.
\]
The expectation form follows from
\[
\Prob(R_{r,n}=m)=\frac{\binom rmS(n,m)m!}{r^n},
\]
and the coefficient form follows from
\[
m!S(n,m)=n![u^n](e^u-1)^m,
\qquad
(s-m)!^{-1}=[z^{s-m}]e^z.
\]
Finally, the stationary laws are pushforwards of \(\pi_Q(g)=f(g)^n/(k!Z_{k,n})\) and
\(\pi_K(x)=(k-r_x)!/(k!Z_{k,n})\), using
\[
|C_s|=\binom{k}{s}\cdot\,!\,(k-s),
\qquad
|P_m|=\binom kmS(n,m)m!. \qedhere
\]
\end{proof}

\begin{corollary}[Fixed-count quotient as \(n\to\infty\)]
\label{cor:fixed-count-asymptotic}
For fixed \(k\ge3\),
\[
\bar Q^{\#}(r,s)\longrightarrow
\bar Q^{\#}_\infty(r,s):=
\begin{cases}
\displaystyle \theta_{k,s}(s-r)!^{-1},&s\ge r,\\[3pt]
0,&s<r,
\end{cases}
\qquad r,s\in I_k.
\]
Thus, ordered by increasing \(r\), \(\bar Q^{\#}_\infty\) is upper triangular
with eigenvalues
\[
\theta_{k,r}=\frac{!\,(k-r)}{(k-r)!},
\qquad r\in I_k.
\]
The largest nontrivial limiting eigenvalue is \(1/2\).
\end{corollary}

\begin{proof}
For fixed \(r\),
\[
\Prob(R_{r,n}<r)\le r\left(1-\frac1r\right)^n\to0,
\]
so \(R_{r,n}\to r\) in probability. Since \((s-R_{r,n})!^{-1}\) is bounded and
is interpreted as \(0\) when \(R_{r,n}>s\), bounded convergence in
\eqref{eq:Qbar-from-Q-expect} gives
\[
\bar Q^{\#}(r,s)\to \theta_{k,s}(s-r)!^{-1},
\]
with value \(0\) when \(s<r\). The limiting matrix is upper triangular, so its
eigenvalues are its diagonal entries. The largest nontrivial one is
\(!2/2!=1/2\), and \(!m/m!<1/2\) for \(m\ge3\).
\end{proof}

\begin{example}[\(k=3,n=2\)]
Here \(C_3=\{e\}\) and \(C_1=\{(12),(13),(23)\}\). Since the four elements of
\(S_3^*\) have distinct fixed sets, the exact fixed-set quotient is the full
dual chain, with
\[
Q_n=
\bordermatrix{
      & e & (12) & (13) & (23) \cr
e     & \frac56 & \frac1{18} & \frac1{18} & \frac1{18} \cr
(12)  & \frac12 & \frac12    & 0          & 0          \cr
(13)  & \frac12 & 0          & \frac12    & 0          \cr
(23)  & \frac12 & 0          & 0          & \frac12
},
\qquad
\Spec(Q_n)=\left\{1,\frac12,\frac12,\frac13\right\}.
\]
The coarser fixed-point-count quotient is
\[
\bar Q^{\#}=
\bordermatrix{
      & C_3 & C_1 \cr
C_3   & \frac56 & \frac16 \cr
C_1   & \frac12 & \frac12
},
\qquad
\Spec(\bar Q^{\#})=\left\{1,\frac13\right\}.
\]
Thus the fixed-set quotient is spectral-complete, while the fixed-point-count
quotient loses the two transverse \(1/2\)-eigenvectors.
\end{example}

\subsection{Mixing in the Value Model}

The elementary pointwise floor behind the universal bound is especially simple in
the value model. By Lemma~\ref{lem:value-basic}, \(|G_x|=(k-r_x)!\), so
\begin{equation}\label{eq:maxGx}
M:=\max_x |G_x|=(k-1)!.
\end{equation}
For all \(u,v\in[k]^n\),
\[
K(u,v)\ge \frac1{|G_u|\,|X_e|}
\ge
\frac1{(k-1)!k^n},
\]
equivalently,
\[
K(u,\cdot)\ge \frac1{(k-1)!}\Unif([k]^n)(\cdot).
\]
If \(n\ge k-1\), this pointwise floor is sharp: take \(u=1^n\) and take \(v\)
using all symbols \(2,\dots,k\). Paguyo~\cite[Lemma~4.4, Theorem~4.5]{Paguyo22} obtains this floor and the corresponding \(n\)-independent minorization bound for \(K\) in the \(k<n\) regime; Theorem~\ref{thm:value-mixing-summary}\textup{(a)} gives the same type of bound for both \(K\) and \(Q\).

For \(k\ge n\), Paguyo's coupling bound \cite[Theorem~1.1]{Paguyo22} and
spectral estimate \cite[Proposition~4.3]{Paguyo22} for the primal chain transfer
to the dual chain with the one-step loss from Corollary~\ref{lem:TV-step}.

\begin{theorem}[Value-model mixing bounds]\label{thm:value-mixing-summary}
For \(S_k\curvearrowright[k]^n\), \(k\ge2\), the following hold.

\begin{enumerate}[label=\textup{(\alph*)},itemsep=2pt]

\item For \(t\ge1\),
\[
d_Q(t),d_K(t)\le \left(1-\frac1{(k-1)!}\right)^t,
\]
and
\[
t_{\mix}(Q;\varepsilon),\ t_{\mix}(K;\varepsilon)
\le
\left\lceil (k-1)!\log\frac1\varepsilon\right\rceil .
\]

\item If \(k\ge n\), then
\begin{equation}\label{eq:value-Q-Paguyo}
d_Q(t)\le n\left(1-\frac1{2k}\right)^{t-1}
\qquad(t\ge1),
\end{equation}
and
\[
t_{\mix}(Q;\varepsilon)
\le
1+\left\lceil 2k\log\frac n\varepsilon\right\rceil .
\]

\item If \(k\ge n\), then
\[
\lambda_1(Q)=\lambda_1(K)\le 1-\frac1{2k}.
\]
Equivalently,
\[
\gap(Q)=\gap(K)\ge\frac1{2k},
\qquad
t_{\mathrm{rel}}(Q)=t_{\mathrm{rel}}(K)\le2k.
\]
\end{enumerate}
\end{theorem}

\begin{proof}
Part \textup{(a)} follows from \eqref{eq:maxGx} and Theorem~\ref{cor:universal-mix}.

Part \textup{(b)} follows from Paguyo's bound
\[
d_K(t)\le n\left(1-\frac1{2k}\right)^t
\qquad(k\ge n)
\]
\cite[Theorem~1.1]{Paguyo22} and Corollary~\ref{lem:TV-step}. 

Part \textup{(c)} follows from Corollary~\ref{cor:gap-equal} and Paguyo's spectral bound
\cite[Proposition~4.3]{Paguyo22}.
\end{proof}

\section{The Coordinate-Permutation Model}\label{sec:coordinate}

This section studies \(S_n\curvearrowright[k]^n\), where \(S_n\) permutes coordinates. The transition \(Q(g,h)\) is governed by the joint orbits of \(\langle g,h\rangle\) on \([n]\). We derive coloring formulas, identify \(\pi_Q\) with Ewens\((k)\), and specialize binary spectral consequences.

\subsection{Setting and Basic Properties}\label{sec:coordinate-basic}

Let \(G=S_n\) act on \(X=[k]^n\) by
\[
(g\cdot x)_i=x_{g^{-1}(i)}.
\]
For \(x\in[k]^n\), define the index sets, multiplicities, and histogram by
\[
I_a(x):=\{i\in[n]:x_i=a\},
\qquad
m_a(x):=|I_a(x)|,
\qquad
\mathbf m(x):=(m_1(x),\dots,m_k(x)).
\]
Thus \(\sum_{a=1}^k m_a(x)=n\). For \(k=2\), writing the alphabet as
\(\{0,1\}\), the Hamming weight is
\[
w(x):=m_1(x)=|\{i:x_i=1\}|.
\]
For \(g\in S_n\), let \(c(g)\) be the total number of cycles.

\begin{lemma}[Basic identities]\label{lem:coord-basic}
For \(g\in S_n\) and \(x\in[k]^n\),
\[
\begin{gathered}
|X_g|=k^{c(g)},\qquad G^*=S_n,\qquad
|[k]^n/S_n|=\binom{n+k-1}{k-1},\\
G_x\cong S_{m_1(x)}\times\cdots\times S_{m_k(x)},
\qquad
|G_x|=\prod_{a=1}^km_a(x)!.
\end{gathered}
\]
Consequently,
\[
Q(g,h)=\frac1{k^{c(g)}}\sum_{x\in X_g\cap X_h}
\frac1{\prod_{a=1}^km_a(x)!}.
\]
\end{lemma}

\begin{proof}
A word fixed by \(g\) is constant on each cycle of \(g\), giving \(k\) choices per cycle. The stabilizer of \(x\) freely permutes coordinates inside each level set \(I_a(x)\). Finally, two words lie in the same orbit iff they have the same histogram, and histograms are counted by stars and bars.
\end{proof}

\subsection{Joint-Orbit Closed Forms}\label{sec:coord-canonical}

For \(g,h\in S_n\), let \(H:=\langle g,h\rangle\), and let \(O_1,\dots,O_s\) be the \(H\)-orbits on \([n]\), with sizes
\[
b_j:=|O_j|,
\qquad
\sum_{j=1}^sb_j=n.
\]
For \(H\le S_n\), write the common fixed set of \(H\) as
\[
X_H:=\{x\in[k]^n:\sigma\cdot x=x\ \text{for all }\sigma\in H\}
=\bigcap_{\sigma\in H}X_\sigma.
\]
If \(H\) has orbits \(O_1,\dots,O_s\) on \([n]\), then
\[
X_H=\{x\in[k]^n:x\text{ is constant on each }O_i\}.
\]

\begin{lemma}[Joint constancy]\label{lem:joint-constancy}
For \(H=\langle g,h\rangle\),
\[
X_g\cap X_h=X_H.
\]
Consequently, if \(H\) has \(s\) orbits on \([n]\), then
\[
|X_g\cap X_h|=|X_H|=k^s.
\]
\end{lemma}

\begin{proof}
We have
\[
x\in X_g\cap X_h
\iff g\cdot x=x,\ h\cdot x=x
\iff \sigma\cdot x=x\ \text{for all }\sigma\in\langle g,h\rangle.
\]
Thus \(X_g\cap X_h=X_H\). In the coordinate action, \(x\in X_H\) iff \(x\) is constant on each \(H\)-orbit. Each orbit has \(k\) possible symbols, so \(|X_H|=k^s\).
\end{proof}

A coloring \(\phi:[s]\to[k]\) assigns a symbol to each orbit. The map
\[
\Phi:X_H\to[k]^{[s]},
\qquad
\Phi(x)(j):=x_i\quad(i\in O_j),
\]
is a bijection, with inverse \(x(\phi)\) defined by \(x_i=\phi(j)\) for \(i\in O_j\). Define the orbit-mass counts
\[
M_a(\phi):=\sum_{j:\phi(j)=a}b_j,
\qquad a=1,\dots,k.
\]
Then
\[
m_a(x(\phi))=M_a(\phi).
\]

\begin{theorem}[Closed forms for \(Q(g,h)\)]\label{thm:Qgh-canonical}
With the notation above,
\begin{align}
Q(g,h)
&=\frac1{k^{c(g)}}\sum_{\phi:[s]\to[k]}
\prod_{a=1}^k\frac1{M_a(\phi)!},\label{eq:Q-colorings}\\
&=\frac{k^{-c(g)}}{n!}\sum_{\phi:[s]\to[k]}
\binom n{M_1(\phi),\dots,M_k(\phi)},\label{eq:Q-multinomial}\\
&=k^{s-c(g)}\E\left[\prod_{a=1}^k\frac1{M_a!}\right],\label{eq:Q-expectation}
\end{align}
where \(M_a=\sum_{j=1}^sb_j\ind_{\{U_j=a\}}\) and \(U_1,\dots,U_s\) are i.i.d. uniform on \([k]\). Also,
\begin{align}
Q(g,h)
&=\frac1{k^{c(g)}}[z_1^0\cdots z_k^0]
\exp\left(\sum_{a=1}^kz_a\right)
\prod_{j=1}^s\sum_{a=1}^kz_a^{-b_j},\label{eq:Q-CT-prime}\\
&=k^{-c(g)}[z_1^0\cdots z_k^0]
\frac{(z_1+\cdots+z_k)^n}{n!}
\prod_{j=1}^s\sum_{a=1}^kz_a^{-b_j}.\label{eq:Q-CT}
\end{align}
\end{theorem}

\begin{proof}
By Lemma~\ref{lem:joint-constancy},
\[
Q(g,h)=\frac1{k^{c(g)}}\sum_{x\in X_H}\frac1{\prod_{a=1}^km_a(x)!}.
\]
Apply Lemma~\ref{lemma:reindexing} with
\[
A=X_H,\qquad B=[k]^{[s]},\qquad \Phi:X_H\to[k]^{[s]}.
\]
Since \(\Phi\) is a bijection and \(m_a(x(\phi))=M_a(\phi)\), we get \eqref{eq:Q-colorings}. Using $\prod_a m_a! = n!/\binom{n}{m_1,\dots,m_k}$ gives \eqref{eq:Q-multinomial}. 

If \(U_1,\dots,U_s\) are i.i.d. uniform on \([k]\), then the random vector \((M_1,\dots,M_k)\) has the same law as \((M_1(\phi),\dots,M_k(\phi))\) for a uniform coloring \(\phi:[s]\to[k]\). Hence
\[
\frac1{k^s}\sum_{\phi:[s]\to[k]}\prod_{a=1}^k\frac1{M_a(\phi)!}
=
\E\left[\prod_{a=1}^k\frac1{M_a!}\right],
\]
which gives \eqref{eq:Q-expectation}.

For the coefficient forms, first note that
\[
\sum_{\phi:[s]\to[k]}z_1^{-M_1(\phi)}\cdots z_k^{-M_k(\phi)}
=
\prod_{j=1}^s\sum_{a=1}^kz_a^{-b_j}.
\]
Using \(1/m!=[z^m]e^z\) in \eqref{eq:Q-colorings},
\[
\begin{aligned}
Q(g,h)
&=
\frac1{k^{c(g)}}\sum_{\phi}
[z_1^{M_1(\phi)}\cdots z_k^{M_k(\phi)}]
\exp\!\left(\sum_{a=1}^kz_a\right)\\
&=
\frac1{k^{c(g)}}[z_1^0\cdots z_k^0]
\exp\!\left(\sum_{a=1}^kz_a\right)
\prod_{j=1}^s\sum_{a=1}^kz_a^{-b_j},
\end{aligned}
\]
which is \eqref{eq:Q-CT-prime}. Similarly, using multinomial coefficient extraction in \eqref{eq:Q-multinomial},
\[
Q(g,h)
=
k^{-c(g)}[z_1^0\cdots z_k^0]
\frac{(z_1+\cdots+z_k)^n}{n!}
\prod_{j=1}^s\sum_{a=1}^kz_a^{-b_j},
\]
which is \eqref{eq:Q-CT}.
\end{proof}

\begin{remark}\label{rem:coord-coeff-remarks}
\leavevmode
\begin{enumerate}[label=\textup{(\roman*)},itemsep=2pt]
\item If \(H=\langle g,h\rangle\) is transitive, then \(s=1\), \(b_1=n\), and
\[
Q(g,h)=\frac{k^{1-c(g)}}{n!}.
\]

\item If \(h=e\), then the joint orbits are the cycles of \(g\), and
\[
Q(g,e)=\E\left[\prod_{a=1}^k\frac1{M_a!}\right],
\]
where the expectation is over a uniform coloring of the cycles of \(g\).

\item The two coefficient forms are equivalent because \(\prod_{j=1}^s\sum_{a=1}^k z_a^{-b_j}\) is homogeneous of total degree \(-n\).
Equivalently, after the substitution \(z_a\mapsto z_a^{-1}\), we may write the product with positive powers:
\[
\prod_{j=1}^s\sum_{a=1}^k z_a^{b_j}.
\]
\end{enumerate}
\end{remark}

\begin{corollary}[Binary specialization]\label{cor:binary-forms}
For \(k=2\),
\[
Q(g,h)=2^{-c(g)}\sum_{J\subseteq[s]}\frac1{S_J!(n-S_J)!}
=\frac{2^{-c(g)}}{n!}\sum_{J\subseteq[s]}\binom n{S_J},
\]
where \(S_J:=\sum_{j\in J}b_j\). Equivalently,
\[
Q(g,h)=2^{s-c(g)}\E\left[\frac1{S!(n-S)!}\right]
=\frac{2^{-c(g)}}{n!}[w^n](1+w)^n\prod_{j=1}^s(1+w^{b_j}),
\]
where \(S=\sum_j b_j\xi_j\) and \(\xi_j\) are independent Bernoulli\((1/2)\).
\end{corollary}

\begin{proof}
This is Theorem~\ref{thm:Qgh-canonical} specialized to \(k=2\). A binary
coloring is a subset \(J\subseteq[s]\), with color masses \(S_J\) and \(n-S_J\),
giving the first two formulas. If \(J\) is uniform, then
\(\xi_j:=\mathbf 1_{\{j\in J\}}\) are independent Bernoulli\((1/2)\) variables
and \(S=\sum_jb_j\xi_j\), giving the expectation form. The coefficient form
follows from \eqref{eq:Q-CT} by setting \(z_1=wz_2\) and extracting the constant
term in \(z_2\).
\end{proof}

\begin{example}[Transitive source gives a flat row]\label{ex:ncycle-flat-row}
If \(g\) is an \(n\)-cycle, then \(H=\langle g,h\rangle\) is transitive for every \(h\in S_n\). By Remark~\ref{rem:coord-coeff-remarks}\textup{(i)},
\[
Q(g,h)=\frac{k^{1-c(g)}}{n!}=\frac1{n!},
\]
since \(c(g)=1\).
\end{example}

\begin{corollary}[Uniform pointwise floor]\label{cor:uniform-floor}
Assume \(k\ge2\). For all \(g,h\in S_n\),
\[
Q(g,h)\ge\frac{k^{1-c(g)}}{n!},
\]
with equality iff \(\langle g,h\rangle\) is transitive on \([n]\).
\end{corollary}

\begin{proof}
In \eqref{eq:Q-colorings}, each constant coloring \(\phi\equiv a\) has
\[
M_a(\phi)=n,\qquad M_b(\phi)=0\quad(b\ne a),
\]
so it contributes \(1/n!\). There are \(k\) constant colorings, hence
\[
Q(g,h)\ge \frac1{k^{c(g)}}\cdot \frac{k}{n!}
=\frac{k^{1-c(g)}}{n!}.
\]
Equality holds exactly when there are no nonconstant colorings, i.e. when \(s=1\), equivalently when \(\langle g,h\rangle\) is transitive.
\end{proof}

\begin{example}[Two joint orbits]\label{ex:two-orbits-kary}
Suppose \(\langle g,h\rangle\) has two orbits of sizes \(a\) and \(n-a\).
There are \(k\) colorings in which the two orbits receive the same color; each
contributes \(1/n!\). There are \(k(k-1)\) colorings in which the two orbits receive distinct colors; each contributes \(1/(a!(n-a)!)\). Hence
\[
Q(g,h)
=
\frac1{k^{c(g)}}\left(\frac{k}{n!}+\frac{k(k-1)}{a!(n-a)!}\right)
=
\frac{k^{1-c(g)}}{n!}\left(1+(k-1)\binom na\right).
\]
\end{example}

\begin{theorem}[Identity to a single \(t\)-cycle]\label{thm:kary-IdToTcycle}
Assume \(k\ge2\). Let \(g=e\), and let \(h\in S_n\) be a single \(t\)-cycle,
\(2\le t\le n\). Set \(m:=n-t\). For \(p\ge1\) and \(s\ge0\), define
\[
\kappa_p(s):=
\sum_{\substack{u_1+\cdots+u_p=s\\ u_i\ge0}}
\prod_{i=1}^p\frac1{(u_i!)^2}.
\]
Then
\begin{align}
Q(e,h)
&=
k^{1-n}\sum_{j=0}^{m}
\binom mj
\frac{(m-j)!}{(t+j)!}\,
\kappa_{k-1}(m-j), \label{eq:kary-mult-A}\\
&=
k^{1-n}\sum_{r=0}^{m}
\binom mr
\frac{r!}{(n-r)!}\,
\kappa_{k-1}(r). \label{eq:kary-mult-B}
\end{align}
For \(k=2\),
\begin{equation}\label{eq:binary-t-cycle}
Q(e,h)=\frac{2^{1-n}}{n!}\binom{2n-t}{n}.
\end{equation}
\end{theorem}

\begin{proof}
Since \(g=e\), we have \(c(g)=n\). The orbits of
\[
H=\langle e,h\rangle=\langle h\rangle
\]
are one block of size \(t\) and \(m=n-t\) singleton blocks, as illustrated in
Fig.~\ref{fig:tcycle-orbits}. Write these orbit sizes as
\[
(b_0,b_1,\dots,b_m)=(t,1,\dots,1).
\]

\begin{figure}[htbp]
    \centering
    \resizebox{.9\linewidth}{!}{%
    \begin{tikzpicture}[>=Latex,every node/.style={font=\small}]
      \node[font=\normalsize] at (-1.30,0.45) {$H=\langle h\rangle=$};

      \fill[a0color!70] (0,0) rectangle (3.20,0.90);
      \draw[line width=0.9pt] (0,0) rectangle (3.20,0.90);

      \foreach \x in {3.20,3.80} {
        \fill[a0color!70] (\x,0) rectangle +(0.60,0.90);
        \draw[line width=0.8pt] (\x,0) rectangle +(0.60,0.90);
      }
      \fill[a0color!70] (4.40,0) rectangle +(0.90,0.90);
      \draw[line width=0.8pt] (4.40,0) rectangle +(0.90,0.90);
      \node[inner sep=1pt] at (4.85,0.45) {$\cdots$};
      \fill[a0color!70] (5.30,0) rectangle +(0.60,0.90);
      \draw[line width=0.8pt] (5.30,0) rectangle +(0.60,0.90);

      \foreach \x in {5.90,6.50} {
        \fill[a1color!70] (\x,0) rectangle +(0.60,0.90);
        \draw[line width=0.8pt] (\x,0) rectangle +(0.60,0.90);
      }
      \foreach \x in {7.10,7.70,8.30} {
        \fill[a2color!68] (\x,0) rectangle +(0.60,0.90);
        \draw[line width=0.8pt] (\x,0) rectangle +(0.60,0.90);
      }
      \fill[white] (8.90,0) rectangle +(0.80,0.90);
      \draw[line width=0.8pt] (8.90,0) rectangle +(0.80,0.90);
      \node at (9.30,0.45) {$\cdots$};
      \fill[akcolor!72] (9.70,0) rectangle +(0.60,0.90);
      \draw[line width=0.8pt] (9.70,0) rectangle +(0.60,0.90);

      \draw[decorate,decoration={brace,amplitude=4pt},a0color,line width=0.9pt]
        (0,1.08) -- (5.90,1.08);
      \draw[decorate,decoration={brace,amplitude=4pt},a1color,line width=0.9pt]
        (5.90,1.08) -- (7.10,1.08);
      \draw[decorate,decoration={brace,amplitude=4pt},a2color,line width=0.9pt]
        (7.10,1.08) -- (8.90,1.08);
      \draw[decorate,decoration={brace,amplitude=4pt},akcolor,line width=0.9pt]
        (9.70,1.08) -- (10.30,1.08);

      \node[text=a0color] at (2.95,1.62) {$t+j=M_{a_0}(\phi)$};
      \node[text=a1color,align=center] at (6.50,1.73)
        {$M_{a_1}(\phi)$\\[-2pt]$=r_1$};
      \node[text=a2color,align=center] at (8.00,1.73)
        {$M_{a_2}(\phi)$\\[-2pt]$=r_2$};
      \node[text=akcolor,align=center] at (10.00,1.73)
        {$M_{a_{k-1}}(\phi)$\\[-2pt]$=r_{k-1}$};

      \draw[<->,line width=0.8pt] (0,-0.28) -- node[fill=white,inner sep=1pt] {$t$} (3.20,-0.28);
      \draw[<->,line width=0.8pt] (3.20,-0.28) -- node[fill=white,inner sep=1pt] {$j$} (5.90,-0.28);
      \draw[<->,line width=0.8pt] (5.90,-0.28) -- node[fill=white,inner sep=1pt]
        {$m-j=:r$} (10.30,-0.28);
      \draw[<->,line width=0.9pt] (3.20,-0.82) -- node[fill=white,inner sep=1pt]
        {$n-t=:m$} (10.30,-0.82);

      \draw[decorate,decoration={brace,mirror,amplitude=4pt},a0color,line width=0.9pt]
        (0,-1.10) -- node[below=5pt,text=a0color] {$a_0:=\phi(0)$} (5.90,-1.10);
      \draw[decorate,decoration={brace,mirror,amplitude=4pt},a1color,line width=0.9pt]
        (5.90,-1.10) -- node[below=5pt,text=a1color] {$a_1$} (7.10,-1.10);
      \draw[decorate,decoration={brace,mirror,amplitude=4pt},a2color,line width=0.9pt]
        (7.10,-1.10) -- node[below=5pt,text=a2color] {$a_2$} (8.90,-1.10);
      \node at (9.30,-1.27) {$\cdots$};
      \draw[decorate,decoration={brace,mirror,amplitude=4pt},akcolor,line width=0.9pt]
        (9.70,-1.10) -- node[below=5pt,text=akcolor] {$a_{k-1}$} (10.30,-1.10);
    \end{tikzpicture}%
    }
    \caption{Orbits of \(\langle h\rangle\) and the color counts \(M_a(\phi)\).}
    \label{fig:tcycle-orbits}
\end{figure}

By \eqref{eq:Q-colorings},
\[
Q(e,h)=
\frac1{k^n}
\sum_{\phi:\{0,1,\dots,m\}\to[k]}
\prod_{a=1}^k\frac1{M_a(\phi)!}.
\]
For a coloring \(\phi\), let
\[
a_0:=\phi(0),
\qquad
j:=|\{i\in\{1,\dots,m\}:\phi(i)=a_0\}|,
\qquad
r:=m-j.
\]
Order the remaining colors increasingly:
\[
[k]\setminus\{a_0\}=\{a_1,\dots,a_{k-1}\}.
\]
Define
\[
r_\ell:=|\{i\in\{1,\dots,m\}:\phi(i)=a_\ell\}|,
\qquad \ell=1,\dots,k-1.
\]
Then
\[
r_1+\cdots+r_{k-1}=r.
\]
For such a coloring,
\[
M_{a_0}(\phi)=t+j,
\qquad
M_{a_\ell}(\phi)=r_\ell,
\]
and hence
\[
\prod_{a=1}^k\frac1{M_a(\phi)!}
=
\frac1{(t+j)!}\prod_{\ell=1}^{k-1}\frac1{r_\ell!}.
\]

Apply Lemma~\ref{lemma:reindexing} to
\[
\Phi:[k]^{\{0,1,\dots,m\}}
\longrightarrow
[k]\times\{0,\dots,m\}\times\mathbb N_0^{k-1},
\qquad
\Phi(\phi)=(a_0,j,r_1,\dots,r_{k-1}).
\]
For fixed \((a_0,j,r_1,\dots,r_{k-1})\), the fiber size is
\[
|\Phi^{-1}(a_0,j,r_1,\dots,r_{k-1})|
=
\binom mj\binom{r}{r_1,\dots,r_{k-1}}
=
\binom mj\frac{r!}{r_1!\cdots r_{k-1}!}.
\]
Therefore
\[
\begin{aligned}
\sum_{\phi}\prod_{a=1}^k\frac1{M_a(\phi)!}
&=
\sum_{a_0\in[k]}\sum_{j=0}^{m}
\binom mj\frac1{(t+j)!}
\sum_{\substack{r_1+\cdots+r_{k-1}=r\\ r_\ell\ge0}}
\frac{r!}{r_1!\cdots r_{k-1}!}
\prod_{\ell=1}^{k-1}\frac1{r_\ell!}\\
&=
k\sum_{j=0}^{m}
\binom mj\frac{r!}{(t+j)!}
\sum_{\substack{r_1+\cdots+r_{k-1}=r\\ r_\ell\ge0}}
\prod_{\ell=1}^{k-1}\frac1{(r_\ell!)^2}\\
&=
k\sum_{j=0}^{m}
\binom mj\frac{(m-j)!}{(t+j)!}\kappa_{k-1}(m-j).
\end{aligned}
\]
Multiplying by \(k^{-n}\) proves \eqref{eq:kary-mult-A}.

Since \(r=m-j\), we have \(t+j=t+m-r=n-r\) and \(\binom mj=\binom mr\), giving
\eqref{eq:kary-mult-B}.

For \(k=2\), \(\kappa_1(r)=1/(r!)^2\). Thus \eqref{eq:kary-mult-B} gives
\[
Q(e,h)=2^{1-n}\sum_{r=0}^{m}\binom mr\frac1{r!(n-r)!}
=
\frac{2^{1-n}}{n!}\sum_{r=0}^{m}\binom mr\binom nr.
\]
By Chu-Vandermonde,
\[
\sum_{r=0}^{m}\binom mr\binom nr
=
\binom{n+m}{m}
=
\binom{2n-t}{n-t}
=
\binom{2n-t}{n}.\qedhere
\]
\end{proof}

\begin{remark}[Two edge cases]\label{rem:tcycle-edge-cases}
If \(t=n\), then \(m=0\), so \eqref{eq:kary-mult-A} gives
\[
Q(e,h)=k^{1-n}\frac{1}{n!}.
\]
This matches the transitive case.

If \(t=n-1\), then \(m=1\), and
\[
\kappa_{k-1}(0)=1,\qquad \kappa_{k-1}(1)=k-1.
\]
Thus
\[
Q(e,h)
=
k^{1-n}\left(\frac{k-1}{(n-1)!}+\frac1{n!}\right).
\]
\end{remark}

\begin{example}[\(k=3,n=4,g=e,h=(123)\)]
Here \(H=\langle e,h\rangle=\langle(123)\rangle\) has orbits
\[
O_1=\{1,2,3\},\qquad O_2=\{4\},
\]
so \(s=2\), \(b_1=3\), \(b_2=1\). Formula \eqref{eq:kary-mult-A}, with \(t=3\), \(m=1\), \(\kappa_2(0)=1\), and \(\kappa_2(1)=2\), gives
\[
Q(e,h)
=
3^{-3}\left(\frac{2}{3!}+\frac1{4!}\right)
=
\frac1{72}.
\]
This agrees with Example~\ref{ex:two-orbits-kary}:
\[
\frac{3^{1-4}}{4!}\left(1+2\binom41\right)=\frac1{72}.
\]
\end{example}

\begin{lemma}[Reversibility ratio in the coordinate model]\label{lem:rev-ratio}
For \(S_n\curvearrowright[k]^n\),
\[
\frac{Q(g,h)}{Q(h,g)}=k^{c(h)-c(g)}
\qquad(g,h\in S_n).
\]
\end{lemma}

\begin{proof}
Use Corollary~\ref{cor:revratio} and \(|X_g|=k^{c(g)}\). Also \(Q(h,g)>0\),
since the constant words lie in \(X_h\cap X_g\).
\end{proof}

\begin{corollary}[Single-cycle diagonal]\label{cor:kary-tcycle-correct}
Assume \(k\ge2\). If \(g\in S_n\) is a single \(t\)-cycle, \(2\le t\le n\), and \(m:=n-t\), then
\[
Q(g,g)=Q(g,e)=k^{t-1}Q(e,g).
\]
Consequently,
\begin{align}
Q(g,g)=Q(g,e)
&=
 k^{-m}\sum_{j=0}^{m}\binom mj
\frac{(m-j)!}{(t+j)!}\kappa_{k-1}(m-j),\label{eq:kary-Qge-A}\\
&=
 k^{-m}\sum_{r=0}^{m}\binom mr
\frac{r!}{(n-r)!}\kappa_{k-1}(r).\label{eq:kary-Qge-B}
\end{align}
In particular, if \(t=n\), then
\[
Q(e,g)=\frac{k^{1-n}}{n!},
\qquad
Q(g,g)=Q(g,e)=\frac1{n!}.
\]
If \(t=n-1\), then
\[
Q(e,g)=k^{1-n}\left(\frac{k-1}{(n-1)!}+\frac1{n!}\right),
\qquad
Q(g,g)=Q(g,e)=
\frac1k\left(\frac{k-1}{(n-1)!}+\frac1{n!}\right).
\]
For \(k=2\),
\[
Q(g,g)=Q(g,e)=\frac{2^{-m}}{n!}\binom{2n-t}{n}.
\]
\end{corollary}

\begin{proof}
By Remark~\ref{rem:Qgg}, \(Q(g,g)=Q(g,e)\). By Lemma~\ref{lem:rev-ratio},
\[
Q(g,e)=k^{c(e)-c(g)}Q(e,g).
\]
Since \(c(e)=n\) and a single \(t\)-cycle has \(c(g)=n-t+1\), we get
\[
Q(g,e)=k^{t-1}Q(e,g).
\]
The displayed formulas follow by multiplying Theorem~\ref{thm:kary-IdToTcycle}
and \eqref{eq:binary-t-cycle} by \(k^{t-1}\).
\end{proof}

\subsection{Stationary Distribution and Ewens Law}\label{subsec:sta-cor}

\begin{theorem}[Dual stationary law in the coordinate model]\label{thm:kary-stationary2}
For \(S_n\curvearrowright[k]^n\),
\[
\pi_Q(g)=\frac{k^{c(g)}}{n!\binom{n+k-1}{k-1}}
=\frac{k^{c(g)}}{k(k+1)\cdots(k+n-1)}.
\]
Thus \(\pi_Q\) is the Ewens distribution on \(S_n\) with parameter \(k\).
\end{theorem}

\begin{proof}
Use Theorem~\ref{thm:universal-pi} and Lemma~\ref{lem:coord-basic}. Recall that the Ewens\((\theta)\) law is \(\theta^{c(g)}/(\theta(\theta+1)\cdots(\theta+n-1))\). Thus \(\pi_Q\) is the Ewens distribution on \(S_n\) with parameter \(\theta=k\).
\end{proof}

\begin{corollary}[Cycle-type stationary law]\label{cor:cycle-type-stationary}
Let \(\lambda\vdash n\), and let \(m_i\) be the number of parts of size \(i\) in \(\lambda\). Set
\[
\ell(\lambda):=\sum_{i\ge1}m_i.
\]
For the conjugacy-lumped dual chain,
\[
\bar\pi_Q(\lambda)
=
\frac{k^{\ell(\lambda)}}{\binom{n+k-1}{k-1}
\prod_{i\ge1}i^{m_i}m_i!}.
\]
\end{corollary}

\begin{proof}
The standard conjugacy-class formula in \(S_n\) gives
\[
|\mathcal C_\lambda|
=
\frac{n!}{\prod_{i\ge1}i^{m_i}m_i!}.
\]
For \(g\in\mathcal C_\lambda\), we have \(c(g)=\ell(\lambda)\), hence \(|X_g|=k^{\ell(\lambda)}\).
Therefore, by Corollary~\ref{cor:lump-Q-TV},
\[
\bar\pi_Q(\lambda)
=
|\mathcal C_\lambda|\frac{k^{\ell(\lambda)}}{n!\binom{n+k-1}{k-1}},
\]
which gives the formula.
\end{proof}

For the primal chain,
\[
\pi_K(x)=\frac{\prod_{a=1}^km_a(x)!}{n!\binom{n+k-1}{k-1}}
=
\frac1{\binom{n+k-1}{k-1}\binom n{\mathbf m(x)}},
\]
where
\[
\binom n{\mathbf m(x)}:=\frac{n!}{m_1(x)!\cdots m_k(x)!}.
\]
In the binary case,
\begin{equation}\label{eq:binary-stationary}
\pi_K(x)=\frac1{(n+1)\binom n{w(x)}},
\qquad
\pi_Q(g)=\frac{2^{c(g)}}{(n+1)!}.
\end{equation}

\begin{lemma}[Extrema of the stationary laws]\label{lem:cor-ext}
Write \(n=km+r\), with \(m\ge0\) and \(0\le r<k\). Then
\[
\pi_{K,\max}=\frac1{\binom{n+k-1}{k-1}},
\qquad
\pi_{K,\min}=
\frac{(m!)^{k-r}((m+1)!)^r}{n!\binom{n+k-1}{k-1}}.
\]
For \(Q\),
\[
\pi_{Q,\min}=\frac{k}{n!\binom{n+k-1}{k-1}},
\qquad
\pi_{Q,\max}=\frac{k^n}{n!\binom{n+k-1}{k-1}},
\qquad
\frac{\pi_{Q,\max}}{\pi_{Q,\min}}=k^{n-1}.
\]
In particular, for \(k=2\),
\[
\pi_{K,\max}=\frac1{n+1},
\qquad
\pi_{K,\min}=\frac1{(n+1)\binom n{\lfloor n/2\rfloor}},
\qquad
\pi_{Q,\min}=\frac2{(n+1)!},
\qquad
\pi_{Q,\max}=\frac{2^n}{(n+1)!}.
\]
\end{lemma}

\begin{proof}
Since
\[
\pi_K(x)=\frac1{\binom{n+k-1}{k-1}\binom n{\mathbf m(x)}},
\]
\(\pi_K\) is maximized when \(\binom n{\mathbf m}=1\), i.e. at \((n,0,\dots,0)\). It is minimized when \(\binom n{\mathbf m}\) is maximized.
If two parts satisfy \(a\ge b+2\), replacing them by \(a-1,b+1\) multiplies \(\binom n{\mathbf m}\) by \(a/(b+1)>1\). Hence the maximum occurs at the balanced histogram with \(r\) parts \(m+1\) and \(k-r\) parts \(m\), giving the displayed \(\pi_{K,\min}\). For \(Q\), use \(\pi_Q(g)\propto k^{c(g)}\), with \(1\le c(g)\le n\). The binary formulas are the case \(k=2\).
\end{proof}

\subsection{Mixing in the Coordinate Model}\label{subsec:be-mixing-Q}

The coordinate model has maximal stabilizer \(n!\), so the universal floor is
weak but always available. The useful uniform estimate comes from Aldous's
coupling bound for \(K\), while fixed-\(k\) transitive starts transfer from
Diaconis's constant-start estimates.

\begin{theorem}[Coordinate-model mixing bounds]\label{thm:coordinate-mixing-summary}
For \(S_n\curvearrowright[k]^n\), \(k\ge2\), the following hold.

\begin{enumerate}[label=\textup{(\alph*)},itemsep=2pt]
\item For \(t\ge1\),
\[
d_Q(t),d_K(t)\le\left(1-\frac1{n!}\right)^t,
\]
and
\[
t_{\mix}(Q;\varepsilon),\ t_{\mix}(K;\varepsilon)
\le
\left\lceil n!\log\frac1\varepsilon\right\rceil .
\]

\item
\begin{equation}\label{eq:Q-aldous}
d_Q(t)\le n\left(1-\frac1k\right)^{t-1}
\qquad(t\ge1),
\end{equation}
and
\[
t_{\mix}(Q;\varepsilon)
\le
1+\left\lceil k\log\frac n\varepsilon\right\rceil .
\]

\item For fixed \(k\), there is \(c_k\in(0,1)\), independent of \(n\), such that
if \(g\) is an \(n\)-cycle, then
\[
d_Q(g,t)
\le (1-c_k)^{t-1}
\qquad(t\ge1).
\]
Hence
\[
t_{\mix}(Q;g,\varepsilon)
\le
1+\left\lceil c_k^{-1}\log\frac1\varepsilon\right\rceil .
\]
\item In the diagonal Bose--Einstein case \(k=n\), there is an absolute constant \(c_{\mathrm{BE}}>0\) such that, for any all-equal word \(x_0=a^n\),
\[
d_K(x_0,\ell)\ge c_{\mathrm{BE}}
\qquad(0\le \ell\le \lfloor\log n\rfloor).
\]
Hence
\[
d_K(\ell)\ge c_{\mathrm{BE}}\quad(0\le \ell\le \lfloor\log n\rfloor),
\qquad
d_Q(t)\ge c_{\mathrm{BE}}\quad(0\le t\le \lfloor\log n\rfloor-1),
\]
and
\[
t_{\mix}(K;c_{\mathrm{BE}}/2)\ge \lfloor\log n\rfloor+1,
\qquad
t_{\mix}(Q;c_{\mathrm{BE}}/2)\ge \lfloor\log n\rfloor .
\]
\end{enumerate}
\end{theorem}

\begin{proof}
By Lemma~\ref{lem:coord-basic},
\[
|G_x|=\prod_{a=1}^k m_a(x)!\le n!,
\]
with equality at constant words. Thus \(\max_x|G_x|=n!\), and \textup{(a)} follows from Theorem~\ref{cor:universal-mix}.

Part \textup{(b)} follows from Aldous's bound
\[
d_K(t)\le n\left(1-\frac1k\right)^t
\]
as stated in \cite[Theorem~2]{Diaconis05}, and Corollary~\ref{lem:TV-step}.

For \textup{(c)}, if \(g\) is an \(n\)-cycle, then
\[
X_g=\{a^n:a\in[k]\}.
\]
Apply Theorem~\ref{lem:ptwise-transfer} to Diaconis's fixed-\(k\) bound from all-equal starts \cite[Theorem~1]{Diaconis05}.

For \textup{(d)}, by Diaconis's \(k=n\) lower bound \cite[Section~4, Proposition]{Diaconis05}, there is an absolute constant \(c_{\mathrm{BE}}>0\), independent of \(n\), such that from an all-equal start \(x_0=a^n\),
\[
d_K(x_0,\ell)\ge c_{\mathrm{BE}}
\qquad
(0\le \ell\le \lfloor\log n\rfloor).
\]
Since \(G_{x_0}=S_n\), Theorem~\ref{lem:ptwise-transfer} gives, for \(t\ge0\),
\[
d_K(x_0,t+1)
\le
\max_{h\in G_{x_0}}d_Q(h,t)
=
d_Q(t).
\]
Thus, if \(0\le t\le \lfloor\log n\rfloor-1\), then
\[
d_Q(t)\ge d_K(x_0,t+1)\ge c_{\mathrm{BE}}.
\]
The mixing-time lower bounds follow immediately from the definitions.
\end{proof}

Assume \(n\ge2\). For \(S_n\curvearrowright\{0,1\}^n\), recall that Diaconis--Zhong~\cite{DiaconisZhong21} prove that from
\(x_0\in\{0^n,1^n\}\),
\begin{equation}\label{eq:DZ-symmetric}
\frac14\left(\frac14\right)^t
\le
d_K(x_0,t)
\le
4\left(\frac14\right)^t.
\end{equation}
Consequently, from \(x_0\in\{0^n,1^n\}\),
\begin{equation}\label{eq:binary-primal-start-mixing-window}
\left\lceil\log_4\frac1\varepsilon\right\rceil-1
\le
t_{\mix}(K;x_0,\varepsilon)
\le
1+\left\lceil\log_4\frac1\varepsilon\right\rceil .
\end{equation}
Thus
\[
t_{\mix}(K;x_0,\varepsilon)
=
\left\lceil\log_4\frac1\varepsilon\right\rceil+O(1),
\]
with constants independent of \(n\).

\begin{theorem}[Binary dual spectrum and bounds]
\label{thm:Q-exact-spectrum}
Assume \(n\ge2\) and consider \(S_n\curvearrowright\{0,1\}^n\). Set
\[
\lambda_m:=\left(\frac{\binom{2m}{m}}{2^{2m}}\right)^2
\qquad(0\le m\le\lfloor n/2\rfloor).
\]
Then:
\begin{enumerate}[label=\textup{(\alph*)},itemsep=2pt]
\item
\[
\Spec_{\ne0}(Q)=\{\lambda_m:0\le m\le\lfloor n/2\rfloor\},
\qquad
\operatorname{mult}_Q(\lambda_m)=\binom n{2m},
\]
and
\[
\operatorname{mult}_Q(0)=n!-2^{n-1}.
\]
Thus \(0\in\Spec(Q)\) iff \(n\ge3\).

\item For \(t\ge1\),
\begin{equation}\label{eq:bin-unif-TV-Q}
d_Q(t)
\le
\frac12\sqrt{(n+1)\binom n{\lfloor n/2\rfloor}}
\left(\frac14\right)^{t-1}.
\end{equation}

\item If \(g\) is an \(n\)-cycle, then, for \(t\ge1\),
\[
d_Q(g,t)\le 4\left(\frac14\right)^{t-1},
\qquad
t_{\mix}(Q;g,\varepsilon)
\le
2+\left\lceil\log_4\frac1\varepsilon\right\rceil .
\]

\item For \(t\ge0\),
\[
d_Q(t)\ge4^{-(t+2)},
\qquad
t_{\mix}(Q;\varepsilon)
\ge
\left\lceil\log_4\frac1\varepsilon\right\rceil-2.
\]

\item For \(t\ge1\),
\[
d_Q(e,t)
\le
\frac12\sqrt{\frac{n+1}{4^n}\binom{2n}{n}-1}
\left(\frac14\right)^{t-1}
\le
\frac12\left(\frac n\pi\right)^{1/4}
\left(\frac14\right)^{t-1}.
\]
\end{enumerate}
\end{theorem}

\begin{proof}
For \textup{(a)}, Theorem~\ref{thm:shared-spectrum} transfers the nonzero spectrum and multiplicities from the binary primal spectrum of Diaconis--Lin--Ram~\cite[Theorem~1.2]{DLR25}. Since
\[
\sum_{m=0}^{\lfloor n/2\rfloor}\binom n{2m}=2^{n-1},
\]
the remaining \(n!-2^{n-1}\) eigenvalues are zero.

For \textup{(b)}, use \(\lambda_1(K)=1/4\). By Lemma~\ref{lem:cor-ext},
\[
\pi_{K,\min}^{-1}=(n+1)\binom n{\lfloor n/2\rfloor}.
\]
Thus \eqref{eq:spectral-tv} gives
\[
d_K(t)
\le
\frac12\sqrt{(n+1)\binom n{\lfloor n/2\rfloor}}
\left(\frac14\right)^t.
\]
Then \(d_Q(t)\le d_K(t-1)\) gives the upper bound for \(Q\).

For \textup{(c)}, \(X_g=\{0^n,1^n\}\), so Theorem~\ref{lem:ptwise-transfer}
and the upper bound in \eqref{eq:DZ-symmetric} give the upper bound; solving gives the mixing bound.

For \textup{(d)}, \(d_K(t+1)\le d_Q(t)\), and the lower bound in \eqref{eq:DZ-symmetric} gives
\[
d_Q(t)\ge d_K(t+1)\ge \frac14\left(\frac14\right)^{t+1}=4^{-(t+2)}.
\]

For \textup{(e)}, by \eqref{eq:Q-AK},
\[
Q^t(e,\cdot)=A_eK^{t-1}B,
\qquad
\pi_Q=\pi_KB.
\]
Here \(A_e=U\) is uniform on \(\{0,1\}^n\). Since \(B\) is stochastic, Proposition~\ref{prop:TV-contraction} and
Proposition~\ref{prop:L2-contraction} give
\[
d_Q(e,t)
\le
\|UK^{t-1}-\pi_K\|_{\TV}
\le
\frac12\left(\frac14\right)^{t-1}\sqrt{\chi^2(U\|\pi_K)}.
\]
Using \eqref{eq:binary-stationary},
\[
\chi^2(U\|\pi_K)
=
\sum_x\frac{U(x)^2}{\pi_K(x)}-1
=
\frac{n+1}{4^n}\sum_{w=0}^n\binom nw^2-1
=
\frac{n+1}{4^n}\binom{2n}{n}-1.
\]
Finally, \(\binom{2n}{n}\le4^n/\sqrt{\pi n}\).
\end{proof}

\begin{remark}[Direct \(Q\) spectral bound]
By Lemma~\ref{lem:cor-ext}, in the binary case
\[
\pi_{Q,\min}=\frac2{(n+1)!}.
\]
Thus \eqref{eq:spectral-tv} gives
\[
d_Q(t)\le
\frac12\sqrt{\frac{(n+1)!}{2}-1}\left(\frac14\right)^t
\le
\frac12\sqrt{\frac{(n+1)!}{2}}\left(\frac14\right)^t.
\]
At the same time \(t\), the transferred bound \eqref{eq:bin-unif-TV-Q} is
sharper than the simplified direct bound
\[
\frac12\sqrt{\frac{(n+1)!}{2}}\left(\frac14\right)^t
\]
iff
\[
32\binom n{\lfloor n/2\rfloor}\le n!,
\]
which holds for all \(n\ge6\). For \(2\le n\le5\), the direct bound is no worse.
\end{remark}

\section{Conclusion}

We carried out a systematic study of the dual Burnside process, establishing its fundamental properties and its connection to the classical Burnside chain. The two chains appear to be different processes on different state spaces, but the primal--dual factorization
\[
K=BA,
\qquad
Q=AB
\]
reveals that they share the same nonzero spectrum, the same relaxation time, and total-variation mixing times differing by at most one step. The stationary laws mirror each other:
\[
\pi_K(x)\propto |G_x|,
\qquad
\pi_Q(g)\propto |X_g|.
\]
Thus the dual perspective gives a new lens for studying classical Burnside
processes and sharpening their mixing-time bounds through transfer between the primal and dual chains. It provides both theoretical insight and practical advantages for sampling under group symmetry, and opens new avenues for understanding and accelerating symmetry-aware Markov chain Monte Carlo.

The dual viewpoint is useful because it can compress the analysis. For \(k\ge2\), in the value-permutation model \(S_k\curvearrowright[k]^n\), the dual kernel has an exact fixed-symbol-set quotient of size \(2^k-k-1\), independent of \(n\), preserving the entire nonzero spectra of both chains. For fixed \(k\ge3\), its limiting eigenvalues are
\[
\theta_{k,r}=\frac{!\,(k-r)}{(k-r)!},
\qquad r\in\{1,\dots,k-2,k\},
\]
with multiplicity \(\binom kr\), and hence
\[
\lambda_1(K_n)=\lambda_1(Q_n)\to\frac12,
\qquad
 t_{\mathrm{rel}}(K_n)=t_{\mathrm{rel}}(Q_n)\to2.
\]
For \(S_3\curvearrowright[3]^n\), the dual conjugacy quotient has two states, while the primal orbit chain has \((3^n+3)/6\) states.

In the coordinate-permutation model \(S_n\curvearrowright[k]^n\), the transition \(Q(g,h)\) is controlled by the joint orbits of \(\langle g,h\rangle\). In the binary case, for \(n\ge2\), known primal spectral results imply a dual zero-eigenvalue multiplicity of \(n!-2^{n-1}\), positive exactly when \(n\ge3\), and the identity-start \(L^2\) bound avoids the factorial prefactor from direct worst-case spectral estimates on \(S_n\).

Uniform stabilizer-preserving covers give a useful criterion for when different Burnside problems share the same dual chain. The parking-function example shows that its dual chain is identical to the dual chain for the coordinate-permutation model on \([n+1]^n\), while the primal
chains differ by a fiberwise lift.

Future work includes sharper TV bounds from the exact fixed-symbol-set quotient, geometric methods such as Diaconis--Stroock~\cite{DiaconisStroock91}, conjugacy-lumped estimates in the coordinate model, and a systematic study of Burnside and dual Burnside chains across the twelvefold way. The value-permutation model \(S_k\curvearrowright[k]^n\) is the distinguishable-to-indistinguishable any-map case, whose orbit count is \(\sum_{r\le k}S(n,r)\). The coordinate-permutation model \(S_n\curvearrowright[k]^n\) is the indistinguishable-to-distinguishable any-map case, whose orbit count is \(\binom{n+k-1}{k-1}\).

Injective and surjective restrictions give natural \(G\)-invariant subactions
when nonempty. If \(Y\subseteq X\) is \(G\)-invariant and nonempty, the restricted dual chain lives on
\[
G_Y^*:=\{g\in G:Y_g\ne\varnothing\},
\qquad
Y_g:=X_g\cap Y,
\]
and is obtained from the same formula with \(X_g\) replaced by \(Y_g\). These
restricted kernels are generally not global rescalings of the unrestricted
ones.

For example, in the value-permutation model \(S_k\curvearrowright[k]^n\), the
injective restriction, nonempty when \(n\le k\), gives
\[
|Y_g|=(f(g))_n:=f(g)(f(g)-1)\cdots(f(g)-n+1),
\]
with \((a)_n=0\) if \(a<n\). The surjective restriction, nonempty when
\(n\ge k\), gives
\[
|Y_e|=k!S(n,k),
\qquad
|Y_g|=0\quad(g\ne e).
\]
In the coordinate-permutation model \(S_n\curvearrowright[k]^n\), the
surjective restriction, nonempty when \(n\ge k\), gives
\[
|Y_g|=k!S(c(g),k),
\]
with \(S(m,k)=0\) for \(m<k\), while the injective restriction, nonempty when
\(n\le k\), gives
\[
|Y_e|=k(k-1)\cdots(k-n+1),
\qquad
|Y_g|=0\quad(g\ne e).
\]
Thus injective and surjective restrictions require separate dual-chain analyses.

\appendix

\section{Value-Permutation Model Examples}\label{app:value-examples}

\begin{example}[Value-permutation model: \(S_5\curvearrowright\{1,\dots,5\}^4\)]
For \(g=(12)\),
\[
\Fix(g)=\{3,4,5\},
\qquad
X_g=\{3,4,5\}^4,
\qquad
|X_g|=3^4.
\]
For \(x=(1,3,3,5)\), \(\supp(x)=\{1,3,5\}\), so
\[
G_x\cong S_{5-3}=S_2,
\qquad
|G_x|=2.
\]
The orbit of $x=(1,3,3,5)$ under $S_5$ corresponds to the partition of the positions by equal symbols, \(\{\{1\},\{2,3\},\{4\}\}\). Thus orbits are set partitions of \([n]\) into at most \(k\) blocks; here \(|[5]^4/S_5|=B_4=15\).
\end{example}

\begin{example}[Exact ternary dual chain]\label{ex:ternary-dual}
Let \(S_3\curvearrowright[3]^n\) act by permuting values. Write \(\tau_a\) for
the transposition fixing \(a\in[3]\). Then
\[
S_3^*=\{e,\tau_1,\tau_2,\tau_3\}.
\]
Set
\[
a_n:=\frac1{2\cdot3^{n-1}},
\qquad
\mu_n:=\frac12-a_n.
\]
Since
\[
X_e=[3]^n,\qquad X_{\tau_a}=\{a^n\},\qquad G_{a^n}=\{e,\tau_a\},
\]
we get, in the order \(e,\tau_1,\tau_2,\tau_3\),
\[
Q_n=
\begin{pmatrix}
1-a_n & a_n/3 & a_n/3 & a_n/3\\[3pt]
1/2 & 1/2 & 0 & 0\\[3pt]
1/2 & 0 & 1/2 & 0\\[3pt]
1/2 & 0 & 0 & 1/2
\end{pmatrix}.
\]
Indeed,
\[
Q_n(\tau_a,e)=Q_n(\tau_a,\tau_a)=\frac12,
\qquad
Q_n(e,\tau_a)=3^{-n}\cdot\frac12=\frac{a_n}{3}.
\]

The stationary law is
\[
\pi_Q(e)=\frac{3^n}{3^n+3},
\qquad
\pi_Q(\tau_a)=\frac1{3^n+3},
\]
since \(\pi_Q(g)\propto f(g)^n\), with weights \(3^n,1,1,1\).

Let
\[
C_3:=\{e\},
\qquad
C_1:=\{\tau_1,\tau_2,\tau_3\}.
\]
The conjugacy quotient is
\[
\bar Q_n=
\begin{pmatrix}
1-a_n&a_n\\[3pt]
1/2&1/2
\end{pmatrix}.
\]
For the spectrum, decompose
\[
\mathbb R^4
=
\{(x,y,y,y)\}
\oplus
\{(0,u_1,u_2,u_3):u_1+u_2+u_3=0\}.
\]
On the second summand \(Q_n\) acts by \(1/2\); on the first summand it acts by
\(\bar Q_n\), whose eigenvalues are \(1\) and \(\mu_n\). Hence the eigenvalues
of \(Q_n\), counted with algebraic multiplicity, are
\[
1,\qquad \frac12,\qquad \frac12,\qquad \mu_n,
\]
and
\[
\Spec(\bar Q_n)=\{1,\mu_n\},
\qquad
0\le\mu_n<\frac12.
\]

Since \(Q_n(e,\cdot)\) is class-constant, conjugacy lumping preserves TV from
\(e\). Also
\[
\delta_{C_3}-\bar\pi=\bar\pi(C_1)(1,-1),
\qquad
(1,-1)\bar Q_n=\mu_n(1,-1),
\qquad
\bar\pi(C_1)=\frac3{3^n+3}.
\]
Therefore, for every \(t\ge1\),
\[
d_{Q_n}(e,t)
=
d_{\bar Q_n}(C_3,t)
=
\frac3{3^n+3}\mu_n^t.
\]

The legs are
\[
A(e,w)=3^{-n},
\qquad
A(\tau_a,w)=\mathbf 1_{\{w=a^n\}},
\]
and
\[
B(w,e)=
\begin{cases}
1/2,&|\supp(w)|=1,\\
1,&|\supp(w)|\ge2,
\end{cases}
\qquad
B(w,\tau_a)=\frac12\mathbf 1_{\{w=a^n\}}.
\]
Thus \(Q_n=AB\).

By Theorem~\ref{thm:shared-spectrum}, the classical Burnside kernel \(K_n\) has
the same nonzero eigenvalues as \(Q_n\). Hence
\[
\lambda_1(K_n)=\lambda_1(Q_n)=\frac12,
\qquad
t_{\mathrm{rel}}(K_n)=t_{\mathrm{rel}}(Q_n)=2.
\]
Finally,
\[
|[3]^n/S_3|
=
\sum_{r=0}^3S(n,r)
=
\frac{3^n+3}{6},
\]
whereas \(Q_n\) has \(4\) states and \(\bar Q_n\) has \(2\).
\end{example}

\begin{example}[Complete matrix decomposition for $k=3,\ n=2$]\label{ex:label2}
\emph{States.} 
\[
G^*=\{e,(12),(13),(23)\},\qquad
X=[3]^2=\{11,12,13,21,22,23,31,32,33\}.
\]
(These orders are used throughout for rows/columns of $Q$, $K$, and for $A$, $B$.)

\emph{Forward leg $A: G^*\to X$ (\(4\times 9\)) and Backward leg $B: X\to G^*$ (\(9\times 4\))}
\[
A=\frac{1}{9}\begin{pmatrix}
1 & 1 & 1 & 1 & 1 & 1 & 1 & 1 & 1 \\[2pt]
0 & 0 & 0 & 0 & 0 & 0 & 0 & 0 & 9 \\[2pt]
0 & 0 & 0 & 0 & 9 & 0 & 0 & 0 & 0 \\[2pt]
9 & 0 & 0 & 0 & 0 & 0 & 0 & 0 & 0
\end{pmatrix}
\qquad
B=\frac{1}{2}\begin{pmatrix}
1 & 0 & 0 & 1 \\[2pt]
2 & 0 & 0 & 0 \\[2pt]
2 & 0 & 0 & 0 \\[2pt]
2 & 0 & 0 & 0 \\[2pt]
1 & 0 & 1 & 0 \\[2pt]
2 & 0 & 0 & 0 \\[2pt]
2 & 0 & 0 & 0 \\[2pt]
2 & 0 & 0 & 0 \\[2pt]
1 & 1 & 0 & 0
\end{pmatrix}
\]

\emph{Dual kernel $Q=AB$ (\(4\times 4\)) and Primal kernel $K=BA$ (\(9\times 9\))}
\[
Q=\frac{1}{18}\begin{pmatrix}
15 & 1 & 1 & 1 \\[3pt]
9 & 9 & 0 & 0 \\[3pt]
9 & 0 & 9 & 0 \\[3pt]
9 & 0 & 0 & 9
\end{pmatrix}
\qquad
K=\frac{1}{18}\begin{pmatrix}
10 & 1 & 1 & 1 & 1 & 1 & 1 & 1 & 1 \\[2pt]
2 & 2 & 2 & 2 & 2 & 2 & 2 & 2 & 2 \\[2pt]
2 & 2 & 2 & 2 & 2 & 2 & 2 & 2 & 2 \\[2pt]
2 & 2 & 2 & 2 & 2 & 2 & 2 & 2 & 2 \\[2pt]
1 & 1 & 1 & 1 & 10 & 1 & 1 & 1 & 1 \\[2pt]
2 & 2 & 2 & 2 & 2 & 2 & 2 & 2 & 2 \\[2pt]
2 & 2 & 2 & 2 & 2 & 2 & 2 & 2 & 2 \\[2pt]
2 & 2 & 2 & 2 & 2 & 2 & 2 & 2 & 2 \\[2pt]
1 & 1 & 1 & 1 & 1 & 1 & 1 & 1 & 10
\end{pmatrix}
\]

\emph{Spectra.} \(\mathrm{Spec}(Q)=\{1,\tfrac12,\tfrac12,\tfrac13\}\), \quad 
\(\mathrm{Spec}(K)=\{1,\tfrac12,\tfrac12,\tfrac13,0,0,0,0,0\}\).
The nonzero spectra of \(Q\) and \(K\) coincide, as predicted by the factorization \(Q=AB,\ K=BA\).
\end{example}

\section{Coordinate-Permutation Model Examples}\label{app:coordinate-examples}

\begin{example}[Coordinate-permutation model: \(S_4\curvearrowright\{1,2,3\}^4\)]
For \(g=(12)(34)\), a fixed word has the form \((a,a,b,b)\), so
\[
|X_g|=3^{c(g)}=3^2.
\]
For \(x=(2,2,1,3)\), the histogram is \((1,2,1)\), hence
\[
G_x\cong S_1\times S_2\times S_1,
\qquad
|G_x|=1!2!1!=2.
\]
Orbits are histograms, equivalently weak \(3\)-compositions of \(4\), so
\[
|[3]^4/S_4|=\binom{4+3-1}{3-1}=15.
\]
\end{example}

\begin{example}[Complete matrix decomposition for $k=2,\ n=3$ ($S_3$ acting on $\{0,1\}^3$)]
\emph{States.}
\[
\begin{aligned}
G^* &= S_3 = \{e,(12),(13),(23),(123),(132)\},\\
X&=\{0,1\}^3=\{000,001,010,011,100,101,110,111\}.
\end{aligned}
\]
(These orders are used throughout for rows/columns of $Q$, $K$, and for $A$, $B$.)

\emph{Forward leg $A:G^*\to X$ (\(6\times 8\)) and Backward leg $B:X\to G^*$ (\(8\times 6\))}
\[
A=\frac{1}{8}\begin{pmatrix}
1 & 1 & 1 & 1 & 1 & 1 & 1 & 1 \\[2pt]
2 & 2 & 0 & 0 & 0 & 0 & 2 & 2 \\[2pt]
2 & 0 & 2 & 0 & 0 & 2 & 0 & 2 \\[2pt]
2 & 0 & 0 & 2 & 2 & 0 & 0 & 2 \\[2pt]
4 & 0 & 0 & 0 & 0 & 0 & 0 & 4 \\[2pt]
4 & 0 & 0 & 0 & 0 & 0 & 0 & 4
\end{pmatrix}
\qquad
B=\frac{1}{6}\begin{pmatrix}
1 & 1 & 1 & 1 & 1 & 1 \\[2pt]
3 & 3 & 0 & 0 & 0 & 0 \\[2pt]
3 & 0 & 3 & 0 & 0 & 0 \\[2pt]
3 & 0 & 0 & 3 & 0 & 0 \\[2pt]
3 & 0 & 0 & 3 & 0 & 0 \\[2pt]
3 & 0 & 3 & 0 & 0 & 0 \\[2pt]
3 & 3 & 0 & 0 & 0 & 0 \\[2pt]
1 & 1 & 1 & 1 & 1 & 1
\end{pmatrix}
\]

\emph{Dual kernel $Q=AB$ (\(6\times 6\)) and Primal kernel $K=BA$ (\(8\times 8\))}
\[
Q=\frac{1}{24}\begin{pmatrix}
10 & 4 & 4 & 4 & 1 & 1 \\[3pt]
8 & 8 & 2 & 2 & 2 & 2 \\[3pt]
8 & 2 & 8 & 2 & 2 & 2 \\[3pt]
8 & 2 & 2 & 8 & 2 & 2 \\[3pt]
4 & 4 & 4 & 4 & 4 & 4 \\[3pt]
4 & 4 & 4 & 4 & 4 & 4
\end{pmatrix}
\qquad
K=\frac{1}{16}\begin{pmatrix}
5 & 1 & 1 & 1 & 1 & 1 & 1 & 5 \\[2pt]
3 & 3 & 1 & 1 & 1 & 1 & 3 & 3 \\[2pt]
3 & 1 & 3 & 1 & 1 & 3 & 1 & 3 \\[2pt]
3 & 1 & 1 & 3 & 3 & 1 & 1 & 3 \\[2pt]
3 & 1 & 1 & 3 & 3 & 1 & 1 & 3 \\[2pt]
3 & 1 & 3 & 1 & 1 & 3 & 1 & 3 \\[2pt]
3 & 3 & 1 & 1 & 1 & 1 & 3 & 3 \\[2pt]
5 & 1 & 1 & 1 & 1 & 1 & 1 & 5
\end{pmatrix}
\]

\emph{Spectra.} \(\mathrm{Spec}(Q)=\{1,\tfrac{1}{4},\tfrac{1}{4},\tfrac{1}{4},0,0\}\),\quad
\(\mathrm{Spec}(K)=\{1,\tfrac{1}{4},\tfrac{1}{4},\tfrac{1}{4},0,0,0,0\}\).
The triple eigenvalue \(\tfrac{1}{4}\) matches the Diaconis--Lin--Ram~\cite{DLR25} description.
\end{example}

\section*{Acknowledgements}

The author thanks Jason Fulman for suggesting the dual Burnside process,
proposing the binary coordinate-permutation formulas, and giving thoughtful
feedback. The author also thanks J. E. Paguyo for valuable discussions and
comments, Persi Diaconis for helpful discussions and historical perspective,
and Dominic Arcona for thoughtful questions during an early presentation of
this work.

\section*{Funding}

Open access funding provided by SCELC, Statewide California Electronic Library Consortium

\section*{Data Availability}

No datasets were generated or analyzed during the current study.

\section*{Declarations}

\noindent\textbf{Conflict of interest.}
The author declares no conflict of interest.

\end{document}